\documentclass[a4paper, preprint,11pt]{elsarticle}
\usepackage[left=20mm,right=20mm, top=23mm, bottom=21mm]{geometry}
\usepackage[utf8]{inputenc}
\usepackage[T1]{fontenc} 
\usepackage{geometry}

\usepackage{amsmath,amssymb,amsthm,breqn,pdflscape}
\allowdisplaybreaks 
\usepackage{bm}
\usepackage{mathrsfs}
\usepackage{euscript}
\usepackage{oldgerm}
\usepackage{enumitem} 

\usepackage{graphicx}
\usepackage{epsfig}
\usepackage{float}
\usepackage{subcaption}
\usepackage{overpic}

\usepackage{booktabs}
\usepackage{multirow}
\usepackage{array}
\usepackage{siunitx}

\usepackage[ruled]{algorithm2e}
\usepackage{color}
\usepackage{lmodern}
\usepackage{textcomp}
\usepackage{url}
\usepackage{enumitem}
\usepackage{listings}
\usepackage{setspace}
\usepackage{slantsc}
\usepackage[nottoc]{tocbibind}
\usepackage{lscape}

\usepackage{lineno}
\usepackage{mathtools}
\usepackage{lineno,hyperref}
\usepackage{rotating}

\usepackage{pst-plot}

\usepackage{makeidx}

\usepackage{hyperref}
\usepackage{verbatim} 
\newtheorem{theorem}{Theorem}[section]

\numberwithin{equation}{section}

\begin{document}
\makeatletter
\def\ps@pprintTitle{%
 \let\@oddhead\@empty
 \let\@evenhead\@empty
 \let\@oddfoot\@empty
 \let\@evenfoot\@empty
}
\makeatother

\begin{frontmatter}

\title{{\bf Explicit Runge-Kutta Methods with  Multiquadric and Inverse Multiquadric Radial Basis Functions }}

\author[IIPE]{Shipra Mahata}
\ead{shipramahata.maths@iipe.ac.in}
\author[IIPE]{Samala Rathan*}
\ead{Corresponding author: Samala Rathan, Email: rathans.math@iipe.ac.in}
\address[IIPE]{Department of Humanities and Sciences, Indian Institute of Petroleum and Energy (IIPE) Visakhapatnam, Andhra Pradesh, 530003, India.}

\begin{abstract}
In this article, a family of two- and three-stage explicit multiquadric (MQ) and inverse multiquadric (IMQ) radial basis functions (RBFs) Runge–Kutta methods are introduced for solving ordinary differential equations. These methods are developed by utilizing MQ- and IMQ-RBF Euler methods. The main advantage of these RBF-based methods lies in their ability to achieve a one-order higher accuracy than their classical Runge-Kutta counterparts without increasing the number of stages. This improvement is made possible by incorporating RBF corrections, where the optimal shape parameter is determined through the local truncation error analysis of the proposed schemes. Convergence and stability analyses, including the study of stability regions, are presented to illustrate how these methods compare with standard Runge–Kutta schemes. Numerical experiments on five benchmark problems further confirm predicted accuracy and stability, demonstrating that MQ- and IMQ-based RBF Runge–Kutta methods provide an alternative to conventional low-stage explicit Runge-Kutta schemes.
\end{abstract}
\begin{keyword}
 Initial value problem \sep Runge-Kutta methods \sep Shape parameter \sep Order of accuracy \sep Radial basis functions \sep Stability.\\
\textit{AMS Classification :} 65L06 
\end{keyword}
\end{frontmatter}
\section{Introduction}
Numerical methods for solving initial value problems (IVPs) of the form
\begin{equation}\label{eq:ivp}
u' = f(t, u), \quad u(t_0) = u_0, \quad a \leq t\leq b,
\end{equation} where $u(t) \in C^{\infty}[a,b]$ and $f(t,u)$ is a class of $C^{\infty}$ function has enormous use in physics, biology, engineering, and control theory, etc.  Among the many available techniques, Runge–Kutta (RK) methods are widely used due to their explicit structure, ease of implementation, and high-order accuracy.
The essential idea behind RK methods is to approximate the solution over a single time step by computing a weighted average of several intermediate slope evaluations of the underlying differential equation. These slope evaluations are taken at carefully chosen points between the current state 
$(t_n, u_n)$ and nect state $(t_{n+1}, u_{n+1})$, By combining this information, RK methods construct an improved estimate of the solution trajectory compared to simpler schemes like Euler’s method. 

\par A classical explicit Runge-Kutta method with $r$ stages can, in general, achieve accuracy of order $r$ when $r \leq 4$ \cite{book1}. However, for $r \geq 5$, there arises a well-known \emph{order barrier}: the relationship between the number of stages and the achievable order of accuracy becomes nonlinear. In particular, beyond fourth order, simply increasing the number of stages within a polynomial-based framework does not automatically yield a corresponding increase in accuracy. The limitation motivates the exploration of alternative formulations and basis functions in the design of higher-order integrators. For implicit Runge-Kutta methods, a $r-$ stage scheme can reach order $2r,$ that goes beyond the order limitations of explicit methods, though at the expense of solving nonlinear systems at each step. Over time, various specialized Runge-Kutta schemes have been developed for different applications: exponential-fitted schemes for oscillatory problems \cite{vanden1999exponentially,simos2001fourth,franco2004exponentially}, symplectic schemes for Hamiltonian systems \cite{sanz1992symplectic,hairer2006structure,mei2017symplectic}, exponential schemes for stiff systems \cite{hochbruck2005explicit,hochbruck2005exponential,hochbruck2010exponential}, total variation diminishing and strong stability preserving schemes for hyperbolic conservation laws \cite{shu1988efficient,gottlieb1998total,gottlieb2001strong,bresten2017explicit}, and implicit–explicit schemes for problems involving both stiff and non-stiff components \cite{pareschi2005implicit,boscarino2016high,SB}. These variants demonstrate the flexibility and broad applicability of the Runge–Kutta framework.

The primary objective of this work is to explore the role of radial basis functions (RBFs) in the construction of accurate time discretization schemes for the numerical solution of differential equations. To set the stage, we first review relevant contributions from the existing literature. In \cite{Jung,Gu}, adaptive RBF-based solvers were introduced for solving initial value problems (IVPs). These solvers reformulated classical time-stepping methods—such as Euler’s method, the midpoint method, and the Adams family of methods—by replacing the traditional polynomial bases with radial basis functions, specifically Multi-Quadratic (MQ) and Gaussian RBFs. The distinguishing feature of such adaptive RBF solvers lies in the presence of a free shape parameter embedded in the basis functions. This parameter is not fixed but can be adaptively tuned according to the local smoothness of the solution, thereby enhancing accuracy and stability. A rigorous strategy for selecting this parameter involves using derivative information to eliminate the leading terms in the local truncation error expansion, effectively improving the order of accuracy of the method. Building upon this foundation, subsequent works such as \cite{Rathan, rathan2024adaptive} extended these ideas further by designing time discretization schemes based on inverse quadratic and inverse multi-quadratic RBFs, thereby broadening the applicability of RBF-based solvers across a wider range of problems. 

\par The shape parameter $\epsilon$ is central to balancing accuracy and stability in RBF-based methods. Since no universal rule exists for its optimal choice, various strategies based on stability and accuracy have been explored \cite{fasshauer2007choosing,fornberg2004some,schaback1995error}, including the use of multiple shape parameters to further improve accuracy. The use of shape parameters to suppress numerical oscillations or errors is not limited to RBF-based ODE solvers. Similar ideas appear in error inhibiting schemes \cite{gu2020adaptive}, Essentially Non-Oscillatory (ENO) and Weighted Essentially Non-Oscillatory (WENO) \cite{guo2017radial,guo2017rbf} methods for solving hyperbolic PDEs, where the reconstruction uses adaptive criteria to control local behavior of the solution. To advance the construction of higher-order time discretization schemes, Gu et al.~\cite{Jiaxi} introduced Runge–Kutta (RK) methods augmented with Gaussian radial basis functions (RBFs). In this framework, RBFs are incorporated into the stage equations of RK methods to enhance their accuracy without increasing the number of stages. The resulting RBF-embedded RK schemes generalize the classical formulation by introducing a free shape parameter at each stage, motivated by the Gaussian RBF modification of Euler’s method \cite{Gu}. The central idea is to systematically reduce the local truncation error by selecting the shape parameter such that the dominant error term in the expansion is canceled, thereby raising the effective order of accuracy by one. The determination of this parameter depends on the local smoothness of the underlying solution, which guides its adaptive adjustment during the integration process.

In this work, we propose a family of two- and three-stage explicit multiquadric (MQ) and inverse multiquadric (IMQ) radial basis function  RK methods for solving ordinary differential equations. These methods are developed by extending MQ- and IMQ-based RBF Euler methods, and are designed to achieve one order higher accuracy than classical RK methods without increasing the number of stages. This enhancement is made possible by incorporating RBF corrections, where the shape parameters are analytically selected to eliminate leading terms in the local truncation error expansion, offering a systematic approach to constructing high-order RBF-based schemes. Unlike prior works \cite{fasshauer2007choosing,fornberg2004some,  schaback1995error} that rely on empirical or numerically optimized shape parameters, this approach determines optimal values through rigorous local error analysis. The proposed framework generalizes beyond Gaussian RBFs, covering both MQ and IMQ variants, and includes detailed convergence and stability analyses such as the study of stability regions to assess the performance relative to standard RK schemes. Numerical experiments on five benchmark problems confirm the predicted accuracy and stability, demonstrating that these RBF-based methods not only match but often exceed the performance of conventional low-stage explicit RK methods. Notably, classical polynomial-based RK methods emerge as special cases of this broader RBF-based formulation, ensuring that the new methods maintain at least the same order of convergence while offering enhanced flexibility and accuracy.



The manuscript is organized as follows; Section 2 reviews classical 
$r$ stage Runge–Kutta methods of order $r\leq 3$ laying the groundwork for comparison with the proposed methods. Section 3 introduces the MQ-RBF and IMQ-RBF Runge–Kutta methods, deriving their formulations from the corresponding RBF-modified Euler methods. We perform local truncation error analysis and establish order conditions to ensure that an  $r$ stage method achieves $(r+1)$-th order accuracy through optimal shape parameter selection. Section 4 presents the convergence analysis of the proposed methods, demonstrating their theoretical consistency and accuracy. Section 5 provides numerical experiments on benchmark problems to validate the improved performance and stability of the RBF-based methods. Finally, some conclusions and future works are provided in Section 6.

\section{Explicit Runge-Kutta Methods}
We denote $u_n$ and $v_n$ as the exact and numerical solutions of the IVP \eqref{eq:ivp}, respectively. The $r$-stage Runge-Kutta methods \cite{book1,iserles2009first} with step length $h$ for \eqref{eq:ivp} can be written as,
\begin{equation}    \label{eq:grk}
v_{n+1} = v_n + h \sum_{i=1}^{r} b_i k_i,
\end{equation} where,
\begin{eqnarray*}
\begin{aligned}
k_1 &= f(t_n, v_n), \\
k_2 &= f(t_n + c_2 h, v_n + h a_{21} k_1), \\
k_3 &= f(t_n + c_3 h, v_n + h (a_{31} k_1 + a_{32} k_2)), \\
&\ \vdots \\
k_r &= f\left(t_n + c_r h, v_n + h \sum_{j=1}^{r-1} a_{rj} k_j\right).
\end{aligned}
\end{eqnarray*}
are slopes at intermediate points $t_n+c_rh$ between $t_n$ to $t_{n+1},$  $a_{rj}$ for $ 1 \leq j <i \leq r$ are coefficients, $b_i$ for $1 \leq i \leq r,$ and $c_r$ for $1 \leq i \leq r.$ The matrix $[a_{rj}]$ is called the Runge-Kutta matrix, $b_i$ are known as weights and $c_r$ are as nodes. Such r-stage explicit Runge-Kutta method can be written in Butcher tableau form \cite{butcher2016numerical} as follows
\begin{table}[h]
    \centering
    \renewcommand{\arraystretch}{1.5}
    \begin{tabular}{c|cccc}
        $c_1=0$  &      &        &        &        \\
        $c_2$   & $a_{21}$  &       &        &        \\
        $c_3$   & $a_{31}$  & $a_{32}$  &       &        \\
        $\vdots$  & $\vdots$  & $\vdots$  & $\ddots$  &        \\
        $c_r$   & $a_{r1}$  & $a_{r2}$  & $\dots$  &  $a_{r(r-1)}$    \\
        \hline
               & $b_1$   & $b_2$   & $\dots$   & $b_r$   
    \end{tabular}
\end{table}
\newline Expanding (\ref{eq:grk}) with Taylor's series and equating the same order coefficients of h for both sides, we get consistency conditions 
\begin{equation}    \label{eq:sumb}
\sum_{j=1}^{r}{b_j} = 1,
\end{equation}
and to get required higher-order accuracy for this work, we assume that it satisfies the following condition;
\begin{equation}    \label{eq:sumbn}
\sum_{j=1}^{i-1}{a_{ij}} = c_i, i = 1,...,r.
\end{equation}
\noindent{\textbf{Two-stage second-order methods:}}
The two-stage Runge-Kutta methods can be written as 
\begin{equation}\label{eq:crk_2}   
\begin{cases}
\begin{aligned}
    v_{n+1} &= v_n + h(b_1 k_1 + b_2 k_2), \\
    k_1 &= f(t_n, v_n), \\
    k_2 &= f(t_n + c_2 h, v_n + h a_{21} k_1),
\end{aligned}
\end{cases}
\end{equation}
and the corresponding Butcher tableau is
\[
  \begin{array}{c|cc}
    0 & \\
    c_2 & a_{21} \\
    \hline
      & b_1 & b_2 \\
  \end{array}
\]
Using the Taylor series expansion, the Runge-Kutta method ({\ref{eq:crk_2}}) become second-order accurate if 
\begin{equation}
  b_2 c_2 = \frac{1}{2}, 
  \label{eq:rk_2_con}
\end{equation}
along with the conditions ({\ref{eq:sumb}}) ($b_1+b_2=1$) and ({\ref{eq:sumbn}}) ($a_{21}=c_2$)  converts two-stage Runge-Kutta methods to a one-parameter family of schemes with $c_2 \neq 0$ as
\begin{eqnarray*}
\begin{cases}
\begin{aligned}
v^{(1)} &= v_n + c_2 h f(t_n, v_n), \\
v_{n+1} &= \left(1 - \frac{1}{c_2}\left(1 - \frac{1}{2c_2}\right)\right) v_n 
+ \frac{1}{c_2}\left(1 - \frac{1}{2c_2}\right) v^{(1)} + \frac{1}{2c_2} h f\left(t_n + c_2 h, v^{(1)}\right) .   
\end{aligned}   
\end{cases}    
\end{eqnarray*}
Butcher tableau representation of the same is 
\begin{equation*}
  \begin{array}{c|cc}
    0 & \\
    c_2 & c_2 \\
    \hline
      & 1 - \frac{1}{2c_2} & \frac{1}{2c_2} \\
  \end{array}
\end{equation*}
\noindent{\textbf{Three-stage third-order methods:}} The three-stage Runge-Kutta methods can be expressed in the form
\begin{equation}    \label{eq:crk_3}
\begin{cases}
\begin{aligned}
v_{n+1} &= v_n + h(b_1 k_1 + b_2 k_2 + b_3 k_3), \\
k_1 &= f(t_n, v_n), \\
k_2 &= f(t_n + c_2 h, v_n + h a_{21} k_1), \\
k_3 &= f(t_n + c_3 h, v_n + h(a_{31} k_1 + a_{32} k_2)).
\end{aligned}
\end{cases}
\end{equation}
Butcher tableau representation of this method is given by 
\[
  \begin{array}{c|ccc}
    0 & \\
    c_2 & a_{21} \\
    c_3 & a_{31} & a_{32} \\
    \hline
      & b_1 & b_2 & b_3 \\
  \end{array}
\]
Alternative form of three-stage Runge-Kutta method is
\begin{eqnarray*}
\begin{cases}
\begin{aligned}
  v^{(1)} &= v_n + c_2 h f(t_n, v_n), \\
    v^{(2)} &= \left(1 - \frac{a_{31}}{c_2}\right)v_n + \frac{a_{31}}{c_2}v^{(1)} + a_{32} h f\left(t_n + c_2 h, v^{(1)}\right), \\
    v_{n+1} &= \left(1 - \frac{1}{c_2}\left(b_1 - \frac{a_{31}}{a_{32}}b_2\right) - \frac{b_2}{a_{32}}\right)v_n + \frac{1}{c_2}\left(b_1 - \frac{a_{31}}{a_{32}}b_2\right)v^{(1)} \\
     & \qquad + \frac{b_2}{a_{32}}v^{(2)} + b_3 h f\left(t_n + c_3 h, v^{(2)}\right).    
\end{aligned}    
\end{cases} 
\end{eqnarray*}
The method will be third-order accurate if,
\begin{equation}\label{eq:rk_3_con}
b_2 c_2 + b_3 c_3 = \frac{1}{2},\,\,\, b_2 c_2^2 + b_3 c_3^2 = \frac{1}{3},\,\, a_{32} b_3 c_2 = \frac{1}{6},
\end{equation}
along with the conditions ({\ref{eq:sumb}}) and ({\ref{eq:sumbn}}). The possible families of three-stage third-order Runge-Kutta methods \cite{butcher2016numerical} are given below.\\

\noindent\textbf{I.} \(c_2 \neq 0\),\, \(c_2 \neq \frac{2}{3}\),\, \(c_3 \neq 0\), \, \(c_2 \neq c_3\).
\[
  \begin{array}{c|ccc}
    0 & \\
    c_2 & c_2 \\
    c_3 & \frac{c_3(3c_3 - 3c_2^2 - c_3)}{c_2(2 - 3c_2)} & \frac{c_3(c_3 - c_2)}{c_2(2 - 3c_2)} \\
    \hline
      & \frac{-3c_3 + 6c_2c_3 + 2 - 3c_3}{6c_2c_3} & \frac{3c_3 - 2}{6c_2(c_3 - c_2)} & \frac{2 - 3c_2}{6c_3(c_3 - c_2)} \\
  \end{array}
\]

\noindent\textbf{II.} \(b_3 \neq 0\),\, \(c_2 = c_3 = \frac{2}{3}\).
\[
  \begin{array}{c|ccc}
    0 & \\
    \frac{2}{3} & \frac{2}{3} \\
    \frac{2}{3} & \frac{2}{3} - \frac{1}{4b_3} & \frac{1}{4b_3} \\
    \hline
      & \frac{1}{4} & \frac{3}{4} - b_3 & b_3 \\
  \end{array}
\]

\noindent\textbf{III.} \(b_3 \neq 0\),\, \(c_2 = \frac{2}{3}\),\, \(c_3 = 0 \).
\[
  \begin{array}{c|ccc}
    0 & \\
    \frac{2}{3} & \frac{2}{3} \\
    0 & - \frac{1}{4b_3}  & \frac{1}{4b_3} \\
    \hline
      & \frac{1}{4} - b_3 & \frac{3}{4}  & b_3 \\
  \end{array}
\]
\section{Explicit  RBF Runge-Kutta Methods}
In this section, we present two variants of Runge-Kutta methods with multi-quadratic (MQ) and inverse multi-quadratic (IMQ) radial basis functions based on respective Euler methods.
\begin{eqnarray}\label{Euler}
\begin{aligned}
\textit{MQ-Euler \cite{Jung}:}\,\, v^{n+1} &= \sqrt{(1 + \epsilon_n^2h^2)}\,\,(v_n + hf_n),\\
\textit{IMQ-Euler \cite{Rathan}:}\,\, v^{n+1} &=  \frac{1} {\sqrt{1 + \epsilon_n^2 h^2}}\,v_n+\sqrt{(1 + \epsilon_n^2h^2)} h f_n,
\end{aligned}    
\end{eqnarray}
where the shape parameter $\epsilon \in \mathbb{R} \cup i\mathbb{R}$. Under special conditions, the square root involved in the above methods may be difficult to evaluate, thus the authors presented remedies \cite{Jung,Rathan}. Such methods read as
\begin{eqnarray}\label{Eulermod}
\begin{aligned}
\textit{MQ-modified Euler \cite{Jung}:}\,\, v^{n+1} &= \left(1 + \dfrac{\epsilon_n^2h^2}{2}\right)\left(v_n + hf_n\right),\\
\textit{IMQ-modified Euler \cite{Rathan}:}\,\, v^{n+1} &= \left(1 - \dfrac{\epsilon_n^2h^2}{2}\right)\left(v_n +(1+\epsilon_n^2h^2) hf_n\right),
\end{aligned}    
\end{eqnarray}
Now, we focus on developing mathematical expressions for the MQ and IMQ-RBF based Runge-Kutta schemes with the building block of Euler methods in \eqref{Euler} rather than in \eqref{Eulermod}. We further analyze truncation errors and derive the order of the method according to the consistency conditions.
\subsection{Two-Stage Third-Order Methods}
In this subsection, we develop two-stage third-order MQ and IMQ-RBF based Runge-Kutta schemes and perform the turncation error analysis.\\ \newline 
\noindent{\textit{(a) MQ-RBF RK scheme:}} The two-stage MQ-RBF Runge-Kutta method can be written as
\begin{eqnarray}\label{mq_rk2}
\begin{cases}
\begin{aligned}
v_{n+1} &= v_n + h(b_1 k_1 + b_2 k_2),\\    
k_1 &= f(t_n, v_n), \\
k_2 &= f\left(t_n + c_2 h,(v_n + hk_1a_{21}) \sqrt{1 + \epsilon_n^2 a_{21}^2 h^2} \right) ,
\end{aligned}    
\end{cases}    
\end{eqnarray}
where $\epsilon_n^2$ is a shape parameter and $\epsilon_n^2 = 0$ reduces the method to the classical form of Runge-Kutta method. Alternative form of the \eqref{mq_rk2} is the following.
\begin{eqnarray*}
\begin{cases}
\begin{aligned}
v^{(1)} &= \sqrt{(1 + \epsilon_n^2 a_{21}^2 h^2)}\,\,(v_n + hk_1a_{21}),\\
v_{n+1} &= v_n\left( 1 - \dfrac{b_1}{a_{21}h}\right) 
    + v^{(1)} \dfrac{b_1 }{\sqrt{1 + \epsilon_n^2 a_{21}^2h^2}} + f(t_n + c_2 h, v^{(1)}) b_2 h.
\end{aligned}
\end{cases}
\end{eqnarray*}
Using a Taylor series expansion, the local truncation error is expressed in powers of  $h.$ Specifically, we expand  $k_2$ using the two-variable Taylor series about the point 
$(t_n, u_n),$ while the derivative of  $u_n$ is obtained from the relation \eqref{eq:ivp} satisfied by 
$u_n.$ Assuming $v_n=u_,n$ the resulting expression represents the error produced after a single time step for \eqref{mq_rk2} is
\begin{align*}
\tau  &  =  \frac{u_{n+1}-u_{n}}{h} - (b_1k_1 + b_2k_2 ), \\
& = \dfrac{1}{h}\left (u_n + hu_n' + \frac{h^2}{2}u_n'' + \frac{h^3}{3!}u_n'''+\cdots - u_n \right)\\
& \qquad - \left(b_1f(t_n, v_n) + b_2f\left(t_n + c_2 h,(v_n + hk_1a_{21}) \sqrt{1 + \epsilon_n^2 a_{21}^2 h^2} \right)\right),\\
& =\left(f +\frac{h}{2}(f_t + ff_u) + \frac{h^2}{3!}(f_{tt} + f_{tu}f + f(f_{tu} + ff_{uu}) + f_u(f_t + ff_u))\right)(t_n,u_n)\\
& \qquad -b_1f(t_n,u_n)- b_2\left( f + c_2hf_t + \left( hfa_{21} - \frac{\epsilon_n^2a_{21}^2h^2u_n}{2} - \frac{\epsilon_n^2a_{21}^2h^2b_2c_2h^3}{2} \right)f_u \right)(t_n,u_n)\\
&\qquad+\mathcal{O}(h^3),\\
& =  \left(1 - b_1 - b_2\right) f(t_n,u_n) + h\left[\left( \frac{1}{2} - b_2 c_2 \right) f_t+ \left( \frac{1}{2} - a_{21} b_2 \right) f_u f \right](t_n,u_n) \\
&\qquad + h^2\bigg[\left( \frac{1}{6} - \frac{1}{2} b_2 c_2^2 \right) f_{tt} + \left( \frac{1}{3} - a_{21} b_2 c_2 \right) f_{tu} f + \left( \frac{1}{6} - \frac{1}{2} a_{21}^2 b_2 \right) f^2f_{uu}\bigg](t_n,u_n) \\
&\qquad + h^2\left(\frac{(f_t + f_u f)f_u}{6} - \frac{\epsilon_n^2 a_{21}^2 b_2 u_n f_u}{2}\right)(t_n,u_n)  +\mathcal{O}(h^3).
\end{align*}
The method will maintain second-order accuracy under the consistency and order conditions \eqref{eq:rk_2_con}, ({\ref{eq:sumb}}), and ({\ref{eq:sumbn}}). In addition, the method will be third order accurate if $a_{21} = c_2 = \frac{2}{3}$, $b_1 = \frac{1}{4}$, $ b_2 = \frac{3}{4}$ and 
\begin{equation*}
    \epsilon_n^2 = \dfrac{f_t + ff_u}{u_n},
\end{equation*}
as the truncation error terms vanish upto $\mathcal{O}(h^3).$\\
\noindent{\textit{(b) IMQ-RBF RK scheme:}}  The two-stage IMQ-RBF Runge-Kutta method is of the form
\begin{eqnarray}\label{eq:imq_2}
\begin{cases}
\begin{aligned}
    v_{n+1} &= v_n + h(b_1 k_1 + b_2 k_2),\\
     k_1 &= f(t_n, v_n),\\
     k_2 &= f\left(t_n + c_2 h, \sqrt{(1 + \epsilon_n^2 a_{21}^2 h^2)} h k_1 + \frac{v_n} {\sqrt{1 + \epsilon_n^2 a_{21}^2 h^2}}\right),
\end{aligned}    
\end{cases}    
\end{eqnarray}
where $\epsilon_n^2$ is the shape parameter and making it to zero, then method \eqref{eq:imq_2} leads to the classical RK-2 method.  This two-stage method can be rewritten as
\begin{eqnarray*}
\begin{cases}
\begin{aligned}
v^{(1)} &= v_n \dfrac{1} {\sqrt{1 + \epsilon_n^2 a_{21}^2 h^2 }}+\sqrt{(1 + \epsilon_n^2 a_{21}^2 h^2)}h k_1,\\
v_{n+1} &=v_n\left( 1 - \frac{b_1}{1 + \epsilon_n^2 a_{21}^2h^2}\right) 
    + v^{(1)}\dfrac{b_1 }{\sqrt{1 + \epsilon_n^2 a_{21}^2h^2}} + b_2 h f\left(t_n + c_2 h, v^{(1)}\right).
\end{aligned}   
\end{cases}    
\end{eqnarray*}
Using Taylor's series expansion, we have local truncation error as
\begin{eqnarray*}
\begin{aligned}
\tau  & =  (1 - b_1 - b_2) f(t_n,u_n) + h\left[\left( \frac{1}{2} - b_2 c_2 \right) f_t + \left( \frac{1}{2} - a_{21} b_2 \right) f_u f \right](t_n,u_n) \\
& \qquad + h^2\biggl[\left( \frac{1}{6} - \frac{1}{2} b_2 c_2^2 \right) f_{tt} + \left( \frac{1}{3} - a_{21} b_2 c_2 \right) f_{tu} f \\
&\qquad + \left( \frac{1}{6} - \frac{1}{2} a_{21}^2 b_2 \right) f^2f_{uu} \biggl](t_n,u_n)  + h^2\bigg[(f_t + f_u f)f_u + \frac{\epsilon_n^2 a_{21}^2 b_2 u_n f_u}{2}\bigg](t_n,u_n)  +\mathcal{O}(h^3).  
\end{aligned}
\end{eqnarray*}
The method will maintain second-order accuracy under the consistency and order conditions \eqref{eq:rk_2_con}, ({\ref{eq:sumb}}), and ({\ref{eq:sumbn}}). In addition, the method will be third order accurate if
The method will be third-order accurate if $a_{21}  = c_2 = \frac{2}{3}$, $b_1 = \dfrac{1}{4}$, $ b_2 = \frac{3}{4}$ and $\epsilon_n^2 = - \dfrac{f_t + ff_u}{u_n}.$

In summary, these two-stage third-order methods can be expressed by the augmented Butcher tableau as follows;
\begin{equation}
    \begin{array}{c|cc|c}
        0 &  &  \\
        c_2=\dfrac{2}{3} & a_{21}=\dfrac{2}{3} &  &\textit{MQ:}\,\epsilon_n^2=\dfrac{f_t + ff_u}{u_n} \\
        \hline
        & b_1=\dfrac{1}{4} & b_2=\dfrac{3}{4}
    \end{array}
\end{equation}
\begin{equation}
\begin{array}{c|cc|c}
        0 &  &  \\
        c_2=\dfrac{2}{3} & a_{21}=\dfrac{2}{3} &  &\textit{IMQ:}\,\epsilon_n^2=- \dfrac{f_t + ff_u}{u_n} \\
        \hline
        & b_1=\dfrac{1}{4} & b_2=\dfrac{3}{4}
    \end{array}
\end{equation}
\subsection{Three-stage fourth-order methods}
In this subsection, we develop three-stage fourth-order MQ and IMQ-RBF based Runge-Kutta schemes and perform the turncation error analysis.\\
\newline
\noindent{\textit{(a) MQ-RBF RK scheme:}} The three-stage MQ-RBF Runge-Kutta scheme is of the form
\begin{eqnarray}\label{eq:mq3}
\begin{cases}
\begin{aligned}
v_{n+1} &= v_n + h(b_1 k_1 + b_2 k_2 + b_3 k_3),\\
k_1 &= f(t_n, v_n),  \\
k_2 &= f\left(t_n + c_2 h,(v_n + hk_1a_{21}) \sqrt{1 + \epsilon_{n2}^2 c_{2}^2 h^2}  \right),\\
k_3 &= f\left(t_n + c_3 h,(v_n + h(k_1a_{31} + k_2a_{32})) \sqrt{1 + \epsilon_{n3}^2 a_{21}^2 h^2}  \right). 
\end{aligned}
\end{cases}   
\end{eqnarray}
The method can be written in alternative form as 
\begin{eqnarray*}
\begin{cases}
\begin{aligned}
 v^{(1)} &= \left(v_n + h f a_{21}\right)\sqrt{1 + \epsilon_{n3}^2 c_3^2 h^2}, \\
    v^{(2)} &= v_n \left(1 - \dfrac{a_{31}}{a_{21}}\right)
    + \left[ \dfrac{a_{31} }{a_{21} \sqrt{1 + \epsilon_{n}^{2} c_2^2 h^2}}v^{(1)}
    + h a_{32} f(t_n + c_2 h, v^{(1)}) \right] \sqrt{1 + \epsilon_{n3}^2 c_3^2 h^2}, \\
    v_{n+1} &= v_n \left[ 1 - b_1 h 
    - \frac{b_2}{a_{32} h \sqrt{1 + \epsilon_{n3}^2 c_3^2 h^2}} 
    + \frac{a_{31}}{a_{32}} \right]\\
 & \qquad+ v^{(1)} \left[ \dfrac{b_1}{a_{21} \sqrt{1 + \epsilon_n^2 a_{21}^2 h^2}} 
    - \frac{b_2 a_{31}}{a_{32} \sqrt{1 + \epsilon_n^2 a_{21}^2 h^2}} \right]+ \dfrac{v^{(2)}}{a_{32} h \sqrt{1 + \epsilon_{n3}^2 c_3^2 h^2}}\\ 
  & \qquad \qquad + b_3 f(t_n + c_3 h, v^{(2)}).
\end{aligned}
\end{cases}   
\end{eqnarray*}
Upon Taylor's series expansion for \eqref{eq:mq3}, the local truncation error can be derived as
\begin{align*}
\tau_n &= \frac{u_{n+1}-u_{n}}{h} - (b_1k_1 + b_2k_2 + b_3k_3), \\
&=  (1 - b_1 - b_2 - b_3)f + h\left[ ff_u\left( \frac{1}{2} -a_{21}b_2 -a_{31}b_3 - a_{32}b_3\right) + f_t\left( \frac{1}{2} -b_3c_3- b_2c_2
 \right)\right] \\
& \quad
 + h^2 \left[ f^2f_{uu}\left( \frac{1}{6} - \frac{a_{21}^2b_2}{2} - \frac{a_{31}^2b_3}{2} - a_{31}a_{32}b_3 - \frac{a_{32}^2b_3}{2}\right) + ff_u^2\left(\frac{1}{6} - a_{21}a_{32}b_3\right) + ff_{ut}\left(\frac{1}{3}-a_{21}b_2c_2 \right.\right.\\
 &\quad
 \left. \left.-a_{31}b_3c_3 - a_{32}b_3c_3\right) + f_tf_u\left( \frac{1}{6} - a_{32}b_3c_2\right) +f_{tt}\left( \frac{1}{6}- \frac{b_2c_2^2}{2} - \frac{b_3c_3^2}{2}\right) - \left(b_2c_2^2 \epsilon_{n2}^2 + b_3c_3^2\epsilon_{n3}^2\right)u_nf_u\right]
 \\
 &\quad
 + 
 h^3 \left[ f^3f_{uuu}\left( \frac{1}{24} -\frac{a_{21}^3b_2}{6}-\frac{a_{31}^3b_3}{6} -\frac{a_{31}^2a_{32}b_3}{2} - \frac{a_{31}a_{32}^2b_3}{2} - \frac{a_{32}^3b_3}{6}\right) + f^2f_uf_{uu}\left( \frac{1}{6} - \frac{a_{21}^2a_{32}b_3}{2} \right.\right. \\
 &\quad
 \left. \left. - a_{21}a_{31}a_{32}b_3 - a_{21}a_{32}^2b_3\right) +f^2f_{uut}\left( \frac{1}{8} -\frac{a_{21}^2b_2c_2}{2} -\frac{a_{31}^2b_3c_3}{2} - a_{31}a_{32}b_3c_3 -\frac{a_{32}^2b_3c_3}{2}\right) + ff_tf_{uu}\left( \frac{1}{8} \right. \right.\\
 &\quad
  \left.- a_{31}a_{32}b_3c_2 - a_{32}^2b_3c_2\right) + ff_{utt}\left(\frac{1}{8} - \frac{a_{31}b_3c_3^2}{2} - \frac{a_{32}b_3c_3^2}{2}\right) + f_{tt}f_u \left(\frac{1}{24} - \frac{a_{32}b_3c_2^2}{2} \right)- f_tf_{ut}\left( \frac{1}{8} \right.
 \\
 &\quad\left.\left. - a_{32}b_3c_2c_3\right)+ f_{ttt}\left( \frac{1}{24} -\frac{b_2c_2^3}{6} - \frac{b_3c_3^3}{6}\right) - \frac{a_{32}b_3c_2^2u_nf_u^2\epsilon_{2n}^2}{2} - \frac{a_{31}b_3c_3^2\epsilon_{3n}^2ff_u}{2} -\frac{a_{31}b_3c_3^2\epsilon_{3n}^2u_nff_u}{2}\right. \\
 & \quad \left.-\frac{a_{32}b_3c_3^2\epsilon_{3n}^2u_nff_{uu}}{2}  - \frac{a_{32}b_3 c_3^2ff_{u}\epsilon_{3n}^2}{2}-\frac{b_2c_2^3u_n\epsilon_{2n}^2f_{ut}}{2} -\frac{b_3c_3^3u_n\epsilon_{3n}^2f_{ut}}{2}
 \right]+\mathcal{O}(h^4).
\end{align*}
As the conditions  \eqref{eq:sumb}, \eqref{eq:rk_3_con} and 
\begin{equation*}
    b_2c_2^2\epsilon_{n2}^2 + b_3c_3^2\epsilon_{n3}^2 = 0,
 \end{equation*}
 are imposed the above local truncation error term vanishes. There are three cases where different parameter values makes the method forth-order accurate.
 \begin{enumerate}[label=\Roman*.]
\item The coefficient of $h^3$ become zero when
\begin{eqnarray*}
\begin{aligned}
& a_{21} = c_2 = \frac{1}{3}, 
a_{31} = -\frac{5}{12}, 
a_{32} = \frac{5}{4}, 
c_3 = \frac{5}{6}, 
b_1 = \frac{1}{10}, 
b_2 = \frac{1}{2}, 
b_3 = \frac{2}{5}.\\
& \epsilon^2_{n2} =\frac{
f f_t f_{uu}
- 3 f f_u^3
- f f_u f_{tu}
- 3 f_t f_u^2
+ f_t f_{tu}
- f_{tt} f_u
}{z
v f f_{uu}
- 2v f_u^2
+ v f_{tu}
+ f f_u
}\\
& \epsilon^2_{n3} = -\frac{1}{5}\epsilon^2_{n2}.
\end{aligned}
\end{eqnarray*}

\item  The coefficient of $h^3$ become zero when 
\begin{eqnarray*}
\begin{aligned}
    & a_{21} = c_2 = 1, 
a_{31} = \frac{1}{4}, 
a_{32} = \frac{1}{4}, 
c_3 = \frac{1}{2}, 
b_1 = \frac{1}{6}, b_2 = \frac{1}{6}, 
b_3 = \frac{2}{3}.\\
& \epsilon^2_{n2} =
\frac{
f f_t f_{uu}
+ f f_u^3
- f f_u f_{tu}
+ f_t f_u^2
+ f_t f_{tu}
- f_{tt} f_u
}{
v f f_{uu}
+ 2v f_u^2
+ v f_{tu}
+ f f_u
}\\
& \epsilon^2_{n3} = -\epsilon^2_{n2}.
\end{aligned}    
\end{eqnarray*}

\item The coefficient of $h^3$ become zero when
\begin{eqnarray*}
\begin{aligned}
& a_{21} = c_2 = \frac{1}{2}, 
a_{31} = 0, 
a_{32} = c_3 = \frac{3}{4},
c_3 = 1, 
b_1 = \frac{2}{9}, 
b_2 = \frac{1}{3}, 
b_3 = \frac{4}{9}.\\
& \epsilon^2_{n2} = 
\frac{
-\frac{1}{3} f^3 f_{uuu}
- f^2 f_{tuu}
- 4f f_u^3
- f f_{ttu}
- 4f_t f_u^2
- \frac{1}{3} f_{ttt}
}{
v f f_{uu}
- 4v f_u^2
+ v f_{tu}
+ f f_u
}\\
& \epsilon^2_{n3} = -\frac{1}{3}\epsilon^2_{n2}.
\end{aligned}
\end{eqnarray*}
\end{enumerate}
\paragraph{(b) IMQ-RBF RK scheme:} The three-stage IMQ-RBF Runge-Kutta scheme is of the form
\begin{eqnarray}    \label{eq:imq3}  
\begin{cases}
\begin{aligned}
v_{n+1} & = v_n + h(b_1 k_1 + b_2 k_2 + b_3 k_3),\\
k_1 &= f(t_n, v_n), \\
k_2 &= f\left(t_n + c_2 h, \sqrt{(1 + \epsilon_{n2}^2 c_{2}^2 h^2)}ha_{21} f_n + \frac{v_n} {\sqrt{1 + \epsilon_{n2}^2 c_{2}^2 h^2}}\right),\\
k_3 &= f\left(t_n + c_3 h, \sqrt{(1 + \epsilon_{n3}^2 c_{3}^2 h^2)}h (a_{31}k_1 + a_{32}k_2)+ \frac{v_n} {\sqrt{1 + \epsilon_{n3}^2 c_{3}^2 h^2}}\right). 
\end{aligned}    
\end{cases}    
\end{eqnarray}
Truncation error is 
\begin{align*}
    \tau_n  = & \frac{u_{n+1}-u_{n}}{h} - (b_1k_1 + b_2k_2 + b_3k_3),\\
      = & (1 -b_1-b_2-b_3)f + h\left[ff_u \left(\frac{1}{2} - a_{21}b_2 -a_{31}b_3 -a_{32}b_3\right) + f_t \left( \frac{1}{2} - b_2c_2 -b_3c_3\right) + \frac{f^2}{2}\right] \\
&\quad 
 +h^2\left[ f^2f_{uu}\left( \frac{1}{6} -\frac{a_{21}^2b_2}{2} - \frac{a_{31}^2b_3}{2} - a_{31}a_{32}b_3 - \frac{a_{32}^2b_3}{2}\right) + ff_u^2 \left( \frac{1}{6}-a_{21}a_{32}b_3\right)  + ff_{ut} \left( \frac{1}{3}-a_{21}b_2c_2 \right. \right.\\
 & \quad
 \left.\left.
 - a_{31}b_3c_3 - a_{32}b_3c_3\right) +f_tf_v\left( \frac{1}{6} - a_{32}b_3c_2\right) +f_{tt}\left( \frac{1}{6} - \frac{b_2c_2^2}{2} - \frac{b_3c_3^2}{2}\right) +\frac{\epsilon_{n2}^2b_2c_2^2u_nf_u}{2}
 \right] + h^3\left[f^3f_{uuu}\left( \right.\right.\\
 & \quad \left. \left.
 \frac{1}{24}-\frac{-a_{21}^3b_2}{6} - \frac{a_{31}^3b_3}{6} - \frac{a_{31}^2a_{32}b_3}{2} - \frac{a_{31}a_{32}^2b_3}{2} - \frac{a_{32}^3b_3}{6}\right) +f^2f_uf_{uu}\left( \frac{1}{6} - \frac{a_{21}^2a_{32}b_3}{2} - a_{21}a_{31}a_{32}b_3 \right.\right.\\
 & \quad \left.\left.
 - a_{21}a_{32}^2b_3\right) + f^2f_{uut} \left( \frac{1}{8}-\frac{a_{21}^2b_2c_2}{2} - \frac{a_{31}^2b_{3}c_3}{2} -a_{31}a_{32}b_3c_3 -\frac{a_{32}^2b_3c_3}{2} \right) + ff_uf_{ut}\left( \frac{5}{24} -a_{21}a_{32}b_3c_2 \right.\right.\\
 & \quad \left.\left. 
 -a_{21}a_{32}b_3c_3\right) +ff_{utt}\left( \frac{1}{8}-\frac{a_{21}b_2c_2^2}{2} - \frac{a_{31}b_3c_3^2}{2} -\frac{a_{32}b_3c_3^2}{2} \right) + ff_tf_{uu}\left( \frac{1}{8} -a_{31}a_{32}b_3c_2 -a_{32}^2b_3c_2\right)\right. \\
 & \quad \left.+ f_{tt}f_u\left( \frac{1}{24} -\frac{a_{32}b_3c_2^2}{2}\right) +f_tf_{ut}\left( \frac{1}{8} - a_{32}b_3c_2c_3\right) + f_{ttt}\left( \frac{1}{24} -\frac{b_2c_2^3}{6} - \frac{b_3c_3^3}{6}\right) + \frac{a_{21}b_2c_2^2u_n\epsilon_{n2}^2ff_{uu}}{2} \right.\\
 &\quad \left. -\frac{a_{21}b_2c_2^2\epsilon_{n2}^2ff_u}{2}\right]+\mathcal{O}(h^4).
\end{align*}
Applying  \eqref{eq:sumb}, \eqref{eq:rk_3_con} and 
\begin{equation*}
    b_2c_2^2\epsilon_{n2}^2 + b_3c_3^2\epsilon_{n3}^2 = 0,
 \end{equation*}
 to the above local truncation error term vanish. Here, we have four cases.
 \begin{enumerate}[label=\Roman*.]
\item  The coefficient of $h^3$ become zero when,
\begin{eqnarray*}
\begin{aligned}
& a_{21} = c_2 = \frac{1}{2}, 
a_{31} = -1, 
a_{32} = 2,
c_3 = 1, 
b_1 = \frac{1}{6}, 
b_2 = \frac{2}{3}, 
b_3 = \frac{1}{6},\\
&  \epsilon^2_{n2} = \frac{-f^2 f_u f_{uu} - f f_t f_{uu} + f f_u^3 - f f_u f_{tu} + f_t f_u^2 - f_t f_{tu}
}{u f f_{uu} - u f_u^2 + u f_{tu} - f f_u},\\
& \epsilon^2_{n3} = -\epsilon^2_{n2}.
\end{aligned}    
\end{eqnarray*}

\item The coefficient of $h^3$ become zero when
\begin{eqnarray*}
\begin{aligned}
& a_{21} = c_2 = \frac{1}{3}, 
a_{31} = -\frac{5}{12}, 
a_{32} = \frac{5}{4}, 
c_3 = \frac{5}{6}, 
b_1 = \frac{1}{10}, 
b_2 = \frac{1}{2}, 
b_3 = \frac{2}{5},\\
& \epsilon^2_{n2} = \frac{  
    -f f_t f_{uu} + 3f f_u^3 + f f_u f_{tu} + 3f_t f_u^2 - f_t f_{tu} + f_{tt} f_u  
}{  
    u f f_{uu} - 2u f_u^2 + u f_{tu} - f f_u  
}  
, \\
& \epsilon^2_{n3} = -\frac{1}{5}\epsilon^2_{n2}.\\
\end{aligned}    
\end{eqnarray*}

\item  The coefficient of $h^3$ become zero when 
\begin{eqnarray*}
\begin{aligned}
& a_{21} = c_2 = 1, 
a_{31} = \frac{1}{4}, 
a_{32} = \frac{1}{4}, 
c_3 = \frac{1}{2}, 
b_1 = \frac{1}{6}, b_2 = \frac{1}{6}, 
b_3 = \frac{2}{3},\\
& \epsilon^2_{n2} =\frac{ -f f_t f_{uu} - f f_u^3 + f f_uf_{tu} - f_t f_u^2 - f_t f_{tu} + f_{tt} f_u  
}{  
    u f f_{uu} + 2v f_u^2 + u f_{tu} - f f_u }   , \\
& \epsilon^2_{n3} = -\epsilon^2_{n2}.
\end{aligned}    
\end{eqnarray*}
\item  The coefficient of $h^3$ become zero when
\begin{eqnarray*}
\begin{aligned}
& a_{21} = c_2 = \frac{1}{2}, 
a_{31} = 0, 
a_{32} = c_3 = \frac{3}{4},
c_3 = 1, 
b_1 = \frac{2}{9}, 
b_2 = \frac{1}{3}, 
b_3 = \frac{4}{9},  \\
& \epsilon^2_{n2} = 
\frac{
    \frac{1}{3}f^3 f_{uuu} + f^2 f_{tuu} + 4f f_u^3 + f f_{ttu} + 4f_t f_u^2 + \frac{1}{3}f_{ttt}
}{
    u f f_{uu} - 4u f_u^2 + u f_{tu} - f f_u
},\\
& \epsilon^2_{n3} = -\frac{1}{3}\epsilon^2_{n2}.
\end{aligned}    
\end{eqnarray*}
\end{enumerate}
\section{Convergence of RBF Runge-Kutta Methods}
Numerical methods derived in previous section can be applied on IVP if there exists a unique solution. IVP of the form {\eqref{eq:ivp}} has a unique solution if $f(t,u)$ satisfies Lipschitz continuity with respect to second-coordinate that is $|f(t,u_1) - f(t,u_2)| \leq L|u_1 - u_2|$, where $L$ is Lipschitz constant. Now, we prove the convergence of above proposed methods. Note that, we solve the IVP on an interval $[a,b]$ which is divided by uniform mesh length $h=\dfrac{b-a}{N}$ as $N+1$ points with mesh grid points $a=t_0 < t_1<...<t_N=b$ where $t_n=a+nh, n=0,1,...,N.$ We also denote $u_n=u(t_n)$ as exact solution and $v_n$ as numerical approximation for the proposed RBF Runge-Kutta schemes.
\begin{theorem}
    If $\epsilon_{n2}^2$
is bounded for all $n = 0, 1, \dots, N - 1$, then the
two-stage MQ-RBF Runge-Kutta method (\ref{mq_rk2}) converges provided the method satisfies (\ref{eq:sumb}), (\ref{eq:sumbn}) and
(\ref{eq:rk_2_con}).
\end{theorem} 
\begin{proof}
Suppose, $E_{n+1}$ be the error between $u_{n+1}$ and $v_{n+1}$ for fixed $t = t_n$. Now,
\begin{align*}
k_1(w) &= f(t_n, w), \\
k_2(w) &= f \left(t_n + c_2h, (w + hk_1a_{21})\sqrt{1 + \epsilon_{n2}^2 c_{2}^2 h^2} \right).
\end{align*}
$u_{n+1}$ can be expressed as 
\begin{align*}
 u_{n+1} = u_n + h (b_1k_1(u_n) + b_2k_2(u_n)) + h \tau_n,
\end{align*}
where $\tau_n$ is truncation error.
Since, $f(t, u)$ satisfies Lipschitz continuity in $u$, so,
\begin{equation*}
|k_1(u_n)- k_1(v_n)| = |f(t_n, u_n)-f(t_n, v_n)| \leq L|u_n - v_n|,
\end{equation*}
and 
\begin{align*}
    |k_2(u_n)- k_2(v_n)| & = \biggl|f\left(t_n,(u_n + ha_{21}k_1(u_n)) \sqrt{1 + \epsilon_{n2}^2 c_{2}^2 h^2}
\right)-f\left(t_n,(v_n + ha_{21}k_1(v_n) \sqrt{1 + \epsilon_{n2}^2 c_{2}^2 h^2}\right)\biggl|, \\
    & \leq L\biggl|\left(u_n +ha_{21}k_1(u_n)\right)\sqrt{(1 + \epsilon_{n2}^2 c_{2}^2 h^2)}  - \left(v_n + ha_{21}k_1(v_n)\right)\sqrt{(1 + \epsilon_{n2}^2 c_{2}^2 h^2)}\biggl|, \\
    & \leq L\biggl|(u_n - v_n)\sqrt{1 +\epsilon_{n2}^2 c_{2}^2 h^2} + a_{21}hL(u_n -v_n)\sqrt{1 + \epsilon_{n2}^2c_2^2h^2}\biggl|, \\
   &  \leq L\left(\sqrt{(1 + \epsilon_{n2}^2 c_{2}^2 h^2)} + a_{21}hL \sqrt{1 + \epsilon_{n2}^2c_2^2h^2}\right)|u_n - v_n |,
\end{align*}
where $L$ is Lipschitz constant. Now, we compute error
\begin{align*}
    E_{n+1} & = |u_{n+1} - v_{n+1}|, \\
& = |(u_n - v_n) + h [ b_1(k_1(u_n) - k_1(v_n)) + b_2(k_2(u_n) - k_2(v_n))] + h\tau_n|, \\
& \leq |u_n -v_n| + h [b_1|k_1(u_n) - k_1(v_n)| + b_2|k_2(u_n)- k_2(v_n)|] + h|\tau_n|, \\
& \leq \left(1 + hLb_1 + b_2(1 + hL)\sqrt{1 + \epsilon_{n2}^2c_2^2h^2}\right)|u_n -v_n| + h|\tau_n|, \\
& = \phi_n E_n + h|\tau_n|,
\end{align*}
with
\begin{align*}
    \phi_n & = 1 + b_1hL + b_2 (1 + a_{21}hL)  \sqrt{1 + \epsilon_{n2}^2 c_{2}^2 h^2},  \\
   & \leq 1 + h L + e^{hL}\mathcal{C}_2h, \\
   & \leq e^{hL} +  e^{hL}\mathcal{C}_2h,\\
   & =  e^{hL}(1 + \mathcal{C}_2h ),   \\
   & = e^{h(L+\mathcal{C}_2)},
\end{align*} where,
\[
\mathcal{C}_2 = \sup_{\substack{0 < h \leq b - a \\ n = 0, \dots, N - 1}} 
\frac{ \sqrt{1 + \epsilon_{n2}^2 c_2^2 h^2}}{h}.
\]
Here, $c_2$ is constant and $\epsilon_{n2}^2$ is bounded as it depends on $u$ and $u$ is bounded ensured by Lipschitz continuity. 
Thus $E_n$ satisfies  
\begin{equation*}
    E_n \leq \sum_{j=0}^{n-1} \varphi_j E_0 + h \sum_{j=1}^{n-1} \left( \sum_{m=j}^{n-1} \varphi_m \right) |\tau_{j-1}| + h |\tau_{n-1}|.
\end{equation*}
Observing,
\begin{equation*}
    \sum_{m=j}^{n-1} \varphi_m \leq \sum_{m=0}^{n-1} \varphi_m \leq e^{h(L + C_2) n} \leq e^{h(L + C_2) N} = e^{(L + C_2)(b - a)},
\end{equation*}
substituting, \( E_0 = 0 \) and
\[
\|\tau\|_\infty = \max_{n = 0, \dots, N-1} |\tau_n|,
\]
it concludes that for every \( n = 0, 1, \dots, N \),
\begin{equation*}
    E_n \leq h e^{(L + C_2) T} \sum_{j=0}^{n-1} \|\tau\|_\infty = n h e^{(L + C_2)(b - a)} \|\tau\|_\infty \leq (b - a) e^{(L + C_2)(b - a)} \|\tau\|_\infty,
\end{equation*}
and hence 
\begin{equation*}
\lim_{\substack{h \to 0 \\ N h = b - a}} E_n = 0.
\end{equation*}
So, $v_n$ converges to $u_n$.
\end{proof} 
\begin{theorem} If $\epsilon_{n2}^2$ and $\epsilon_{n3}^2$
are bounded for all $n = 0, 1, \dots, N - 1$. then the
three-stage MQ-RBF Runge-Kutta method (\ref{eq:mq3}), converges provided the method satisfies (\ref{eq:sumb}), (\ref{eq:sumbn}) and
(\ref{eq:rk_3_con}).
\end{theorem}
\begin{proof} Suppose, $E_{n+1}$ be the error between $u_{n+1}$ and $v_{n+1}$ for fix $t = t_n$. Now,
\begin{align*}
k_1(w) &= f(t_n, w), \\
k_2(w) &= f(t_n + c_2h, (w + hk_1a_{21})\sqrt{1 + \epsilon_{n2}^2 c_{2}^2 h^2} ),\\
k_3(w) &= f(t_n + c_3h, (w + h(k_1a_{31} + k_2a_{32}))\sqrt{1 + \epsilon_{n3}^2 c_{3}^2 h^2} ).
\end{align*}
    $u_{n+1}$ can be expressed as , 
\begin{align*}
 u_{n+1} = u_n + h (b_1k_1(u_n) + b_2k_2(u_n) + b_3k_3(u_n)) + h \tau_n,
\end{align*}
where $\tau_n$ is truncation error. Since, 
 $f(t, u)$ satisfies Lipschitz continuity with respect to  $u$, so,
\begin{equation*}
|k_1(u_n)- k_1(v_n)| = |f(t_n, u_n)-f(t_n, v_n)| \leq L|u_n - v_n|,
\end{equation*}
\begin{equation*}
|k_2(u_n)- k_2(v_n)| \leq L\left (\sqrt{1 + \epsilon_{n2}^2 c_{2}^2 h^2} + a_{21}hL \sqrt{1 + \epsilon_{n2}^2c_2^2h^2} \right) \left|u_n - v_n \right|,
\end{equation*}
\begin{align*}
    |k_3(u_n)- k_3(v_n)| & = \bigg|f \left(t_n,\left ( u_n + h(a_{31}k_1(u_n) + a_{32}k_2(u_n)\right) \sqrt{1 + \epsilon_{n3}^2 c_{3}^2 h^2}
\right)\\
&\qquad -f\left(t_n,(v_n + h(a_{31}k_1(v_n) + a_{32}k_2(v_n)) \sqrt{1 + \epsilon_{n3}^2 c_{3}^2 h^2}\right)\bigg|, \\
    & \leq L\left|(u_n - v_n) + \left( ha_{31} L|u_n - v_n| + ha_{32}L(1 + hL)\sqrt{1 + \epsilon_{n2}^2c_2^2h^2}|u_n-v_n|\sqrt{1 + \epsilon_{n3}^2c_3^2h^2} \right) \right|, \\
   &  \leq L \left(1 + \left( ha_{31} + ha_{32}(1 + hL)\sqrt{1 + \epsilon_{n2}^2c_2^2h^2}L\sqrt{1 + \epsilon_{n3}^2c_3^2h^2}\right) \right) |u_n - v_n |,
\end{align*}
where $L$ is Lipschitz constant.
\begin{align*}
    E_{n+1} & = |u_{n+1} - v_{n+1}|, \\
    & = |(u_n - v_n) + h [ b_1(k_1(u_n) - k_1(v_n)) + b_2(k_2(u_n) - k_2(v_n)) + b_3(k_3(u_n) - k_3(v_n)) ] + h\tau_n|, \\
    & \leq |u_n - v_n| + h \left[ b_1|k_1(u_n) - k_1(v_n)| + b_2|k_2(u_n) - k_2(v_n)| + b_3|k_3(u_n) - k_3(v_n)| \right] + h|\tau_n|, \\
    & \leq \left( 1 + hLb_1 + b_2(1 + a_{21}hL)\sqrt{1 + \epsilon_{n2}^2 c_2^2 h^2} \right. \\
    & \quad + \left. b_3 hL \left( 1 + h \left( a_{31} + ha_{32}(1 + hL) \sqrt{1 + \epsilon_{n2}^2 c_2^2 h^2} \right) \right) \sqrt{1 + \epsilon_{n3}^2 c_3^2 h^2} \right) \left| u_n - v_n \right| + h|\tau_n|, \\
    & = \phi_n E_n + h|\tau_n|.
\end{align*}
Here,
\begin{align*}
    \phi_n & = 1 + b_1hL + b_2 (1 + hL)  \sqrt{1 + \epsilon_{n2}^2 c_{2}^2 h^2} + b_3hL\left( 1 + (hLa_{31} + hLa_{32}(1 +hL))\sqrt{1 + \epsilon_{n3}^2c_3^2h^2}    \right),\\
   & \leq 1 + h L + b_2hL( 1 + hL)\sqrt{1 + \epsilon_{n2}^2c_2^2h^2} + b_3hL\left(1 + hLa_{31} + hLa_{32}(1+hL)\sqrt{1 + \epsilon_{n3}^2c_3^2h^2}\right),\\
   & \leq e^{hL} + b_2hLe^{hL}\sqrt{1+\epsilon_{n2}^2c_2^2h^2} +b_3hL{e^{hL}+ hLe^{hL}}\sqrt{1 + \epsilon_{n3}^2c_3^2h^2},\\
   & \leq e^{hL} \biggl(1 + b_2hL \sqrt{1 + \epsilon_{n2}^2c_2^2h^2}\biggl)+ b_3hL(1 + hL)\sqrt{ 1 
   + \epsilon_{n3}^2c_3^2h^2}, \\
   & \leq  e^{hL}(1 + \mathcal{C}_3h ),   \\
   & \leq e^{h(L+\mathcal{C}_3)},
\end{align*} with,
\[
\mathcal{C}_3 = \sup_{\substack{0 < h \leq b - a \\ n = 0, \dots, N - 1}} 
\frac{ b_2L\sqrt{1 + \epsilon_{n2}^2 c_2^2 h^2} + b_3L (1+hL)\sqrt{1+\epsilon_{n3}^2c_3^2h^2}}{h}.
\]
Here, $c_3$ is constant and $\epsilon_{n2}^2$ and $\epsilon_{n3}^2$ are bounded as it depends on $u$ and $u$ is bounded ensured by Lipschitz continuity.
Therefore $E_n$ satisfies
\begin{equation*}
    E_n \leq \sum_{j=0}^{n-1} \varphi_j E_0 + h \sum_{j=1}^{n-1} \left( \sum_{m=j}^{n-1} \varphi_m \right) |\tau_{j-1}| + h |\tau_{n-1}|.
\end{equation*}
Clearly,
\begin{equation*}
    \sum_{m=j}^{n-1} \varphi_m \leq \sum_{m=0}^{n-1} \varphi_m \leq e^{h(L + C_2) n} \leq e^{h(L + C_2) N} = e^{(L + C_2)(b - a)}.
\end{equation*}
Upon considering  \( E_0 = 0 \) and
\[
\|\tau\|_\infty = \max_{n = 0, \dots, N-1} |\tau_n|,
\]
it can be concluded that for every \( n = 0, 1, \dots, N \),
\begin{equation*}
    E_n \leq h e^{(L + C_2) T} \sum_{j=0}^{n-1} \|\tau\|_\infty = n h e^{(L + C_2)(b - a)} \|\tau\|_\infty \leq (b - a) e^{(L + C_2)(b - a)} \|\tau\|_\infty,
\end{equation*}
and hence 
\begin{equation*}
\lim_{\substack{h \to 0 \\ N h = b - a}} E_n = 0.
\end{equation*}
\end{proof}
\begin{theorem} If $\epsilon_{n2}^2$
is bounded for all $n = 0, 1, \dots, N - 1$. Then the
two-stage IMQ-RBF Runge-Kutta method ({\ref{eq:imq_2}}) converges ({\ref{eq:sumb}}), provided it satisfies ({\ref{eq:sumbn}}) and
({\ref{eq:rk_2_con}}).
\end{theorem}
\begin{proof} Suppose, $E_{n+1}$ be the error between $u_{n+1}$ and $v_{n+1}$ for fix $t = t_n$. Let
\begin{align*}
k_1(w) &= f(t_n, w), \\
k_2(w) &= f\biggl(t_n + c_2h, \sqrt{(1 + \epsilon_{n2}^2 c_{2}^2 h^2)} \,\, h c_2 k_1 + \frac{w} {\sqrt{1 + \epsilon_{n2}^2 c_{2}^2 h^2}}\biggl).
\end{align*}
 $u_{n+1}$ can be expressed as,
\begin{align*}
 u_{n+1} = u_n + h (b_1k_1(u_n) + b_2k_2(u_n)) + h \tau_n,
\end{align*}
where $ \tau_n$ be the trancation error.
Since  $f(t, u)$ satisfies Lipschitz continuity with respect to $u$,
\begin{equation*}
|k_1(u_n)- k_1(v_n)| = |f(t_n, u_n)-f(t_n, v_n)| \leq L|u_n - v_n|,
\end{equation*}
and 
\begin{align*}
    |k_2(u_n)- k_2(v_n)| & = \Biggl|f\left(t_n, \sqrt{(1 + \epsilon_{n2}^2 c_{2}^2 h^2)} \,\, h c_2 k_1 (u_n) + \frac{u_n} {\sqrt{1 + \epsilon_{n2}^2 c_{2}^2 h^2}}\right) \\
    &\qquad -f \left(t_n, \sqrt{(1 + \epsilon_{n2}^2 c_{2}^2 h^2)} \,\, h c_2 k_1(v_n) + \frac{v_n} {\sqrt{1 + \epsilon_{n2}^2 c_{2}^2 h^2}}\right)\Biggl|, \\
    & \leq L\Biggl|\sqrt{(1 + \epsilon_{n2}^2 c_{2}^2 h^2)}\,\,  h c_2 f(u_n) + \frac{u_n} {\sqrt{1 + \epsilon_{n2}^2 c_{2}^2 h^2}} - \sqrt{(1 + \epsilon_{n2}^2 c_{2}^2 h^2)}\,\,  h c_2 f (v_n)\\
    & - \frac{v_n} {\sqrt{1 + \epsilon_{n2}^2 c_{2}^2 h^2}} \Biggl|, \\
    & \leq L\biggl|\sqrt{(1 + \epsilon_{n2}^2 c_{2}^2 h^2)} \,\, h c_2 (f(u_n) - f(v_n)) + \frac{u_n - v_n} {\sqrt{1 + \epsilon_{n2}^2 c_{2}^2 h^2}}\biggl|, \\
   &  \leq L\biggl|\sqrt{(1 + \epsilon_{n2}^2 c_{2}^2 h^2)} \,\, h c_2 L + \frac{1} {\sqrt{1 + \epsilon_{n2}^2 c_{2}^2 h^2}}\biggl||u_n - v_n|,
\end{align*}
where $L$ be Lipschitz constant. Now,
\begin{align*}
    E_{n+1} & = |u_{n+1} - v_{n+1}|, \\
& = |(u_n - v_n) + h [ b_1(k_1(u_n) - k_1(v_n)) + b_2(k_2(u_n) - k_2(v_n))] + h\tau_n|, \\
& \leq |u_n -v_n| + h [b_1|k_1(u_n) - k_1(v_n)| + b_2|k_2(u_n)- k_2(v_n)|] + h|\tau_n|, \\
& \leq|u_n -v_n| + hb_1 L|u_n -v_n| + h b_2 L\left| \sqrt{(1 + \epsilon_{n2}^2 c_{2}^2 h^2)}\,\,  h c_2 L + \frac{1}{\sqrt{(1 + \epsilon_{n2}^2 c_{2}^2 h^2)}}\right||u_n -v_n|  + h|\tau_n|, \\
& = \phi_n E_n + h|\tau_n|
\end{align*}
with the expression
\begin{align*}
    \phi_n & = 1 + b_1hL + hb_2 L \left( \sqrt{1 + \epsilon_{n2}^2 c_{2}^2 h^2} \,\, L h c_2  + \frac{1}{\sqrt{(1 + \epsilon_{n2}^2 c_{2}^2 h^2)}} \right), \\
   & \leq 1 + b_1 h L + h b_2 L \left(\sqrt{1 + \epsilon_{n2}^2 c_{2}^2h^2}\,\, L h c_2 + 1\right), \\
   & = 1 + b_1 h L + h b_2 L \sqrt{1 + \epsilon_{n2}^2 c_{2}^2h^2}\,\,L ha_{21} + hb_2L, \\
   & = 1 + hb_1L + hb_2L\sqrt{1 + \epsilon_{n2}^2 c_{2}^2h^2}L ha_{21} + hL(1 - b_1), \\
   & = 1 + hL + hb_2L\sqrt{1 + \epsilon_{n2}^2 c_{2}^2h^2} \,\,Lh a_{21} ,
\end{align*}
where,
\[
\mathcal{C}_2 = \sup_{\substack{0 < h \leq b - a \\ n = 0, \dots, N - 1}} 
\frac{h L^2 b_2 a_{21} \sqrt{1 + \epsilon_{n2}^2 c_2^2 h^2}}{e^{hL}}. 
\]
Here, $c_2$ is constant and $\epsilon_{n2}^2$ is bounded as it depends on $u$ and $u$ is bounded ensured by Lipschitz continuity of $f$. Thus,
\begin{align*}
    \phi_n & \leq e^{hL} + \mathcal{C}_2 h e^{hL}, \\
      & \leq (1 + \mathcal{C}_2 h) e^{hL}, \\
      & \leq e^{h(L+\mathcal{C}_2)}.
\end{align*}
Hence $E_n$ satisfies, 
\begin{equation*}
    E_n \leq \sum_{j=0}^{n-1} \varphi_j E_0 + h \sum_{j=1}^{n-1} \left( \sum_{m=j}^{n-1} \varphi_m \right) |\tau_{j-1}| + h |\tau_{n-1}|.
\end{equation*}
Here,
\begin{equation*}
    \sum_{m=j}^{n-1} \varphi_m \leq \sum_{m=0}^{n-1} \varphi_m \leq e^{h(L + C_2) n} \leq e^{h(L + C_2) N} = e^{(L + C_2)(b - a)}.
\end{equation*}
Assuming, \( E_0 = 0 \) and
\[
\|\tau\|_\infty = \max_{n = 0, \dots, N-1} |\tau_n|,
\]
it follows that for every \( n = 0, 1, \dots, N \),
\begin{equation*}
    E_n \leq h e^{(L + C_2) T} \sum_{j=0}^{n-1} \|\tau\|_\infty = n h e^{(L + C_2)(b - a)} \|\tau\|_\infty \leq (b - a) e^{(L + C_2)(b - a)} \|\tau\|_\infty.
\end{equation*}
Hence,
\begin{equation*}
\lim_{\substack{h \to 0 \\ N h = b - a}} E_n = 0.
\end{equation*}
\end{proof}
\begin{theorem} If $\epsilon_{n2}^2$ and $\epsilon_{n3}^2$
are bounded for all $n = 0, 1, \dots, N - 1$. Then the
three-stage RBF Runge-Kutta method (\ref{eq:imq3}), converges provided the method satisfies (\ref{eq:sumb}), (\ref{eq:sumbn}) and
(\ref{eq:rk_3_con}).
\end{theorem}
\begin{proof} Suppose, $E_{n+1}$ be the error $E_{n+1}$  between $u_{n+1}$ and $v_{n+1}$ for fix $t = t_n$. Now,
\begin{align*}
k_1(w) &= f(t_n, w), \\
k_2(w) &= f\left(t_n + c_2h, \sqrt{(1 + \epsilon_{n2}^2 a_{21}^2 h^2)} \,  h a_{21}k_1 + \frac{w} {\sqrt{1 + \epsilon_{n2}^2 a_{21}^2 h^2}}\right), \\
k_3(w) &= f\left(t_n + c_3h, \sqrt{(1 + \epsilon_{n3}^2 c_{3}^2 h^2)} \, h (a_{31}k_1 + a_{32}k_2)+\frac{w} {\sqrt{1 + \epsilon_{n3}^2 c_3^2 h^2}}\right).
\end{align*}
$u_{n+1}$ can be expressed as
\begin{align*}
 u_{n+1} = u_n + h (b_1k_1(u_n) + b_2k_2(u_n) + b_3k_3(u_n)) + h \tau_n,
\end{align*}
where, $\tau_n$ is trancation error.
 Since $f(t,u)$ is Lipschitz continuous in u , so
 \begin{equation*}
     |k_1(u_n) - k_1(v_n)|\leq L|u_n - v_n|,
 \end{equation*}

\begin{equation*}
     |k_2(u_n) - k_2(v_n)|\leq L\left|\sqrt{1 + \epsilon_{n2}^2a_{21}^2h^2}\,\,hk_1a_{21} +\frac{1}{\sqrt{1 +\epsilon_{n2}^2a_{21}^2h^2}}\right||u_n - v_n|,
 \end{equation*}
 \begin{align*}
|k_3(u_n) - k_3(v_n)| 
&= \Bigg| 
f\left(t_n + c_3 h,\ \sqrt{1 + \epsilon_{n3}^2 c_3^2 h^2} \,\, h (a_{31} k_1(u_n) + a_{32} k_2(u_n)) 
+ \frac{u_n}{\sqrt{1 + \epsilon_{n3}^2 c_3^2 h^2}} \right) \\
&\qquad - 
f\left(t_n + c_3 h,\ \sqrt{1 + \epsilon_{n3}^2 c_3^2 h^2} \,\,  h (a_{31} k_1(v_n) + a_{32} k_2(v_n)) 
+ \frac{v_n}{\sqrt{1 + \epsilon_{n3}^2 c_3^2 h^2}} \right)
\Bigg|, \\
&\leq L \Bigg| 
\sqrt{1 + \epsilon_{n3}^2 c_3^2 h^2}  h \left( a_{31}(k_1(u_n) - k_1(v_n)) 
+ a_{32}(k_2(u_n) - k_2(v_n)) \right) 
+ \frac{u_n - v_n}{\sqrt{1 + \epsilon_{n3}^2 c_3^2 h^2}} 
\Bigg|, \\
&\leq L \left[ 
\sqrt{1 + \epsilon_{n3}^2 c_3^2 h^2} \,\, h a_{31} L |u_n - v_n| 
+ \sqrt{1 + \epsilon_{n2}^2 a_{21}^2 h^2} \,\,  h a_{32} L |u_n - v_n| \right. \\
&\quad \left. +  \frac{L|u_n - v_n|}{\sqrt{1 + \epsilon_{n2}^2 a_{21}^2 h^2}}
+ \frac{|u_n - v_n|}{\sqrt{1 + \epsilon_{n3}^2 c_3^2 h^2}} 
\right],
\end{align*}
where, $L$ is Lipschitz constant. 
Now,
\begin{align*}
    E_{n+1} & = |u_{n+1} - v_{n+1}|, \\
& = |(u_n - v_n) + h \left[ b_1(k_1(u_n) - k_1(v_n)) + b_2(k_2(u_n) - k_2(v_n))\right ] + h\tau_n|, \\
& \leq |u_n - v_n| + hb_1L|u_n - v_n| + b_2hL\left(\sqrt{1 + \epsilon_{n2}^2a_{21}^2h^2}\,\,hLa_{21} +  \frac{1}{\sqrt{1 +\epsilon_{n2}^2a_{21}^2h^2}}\right)|u_n - v_n| \\
& + b_3hL \Biggr[ \sqrt{1 + \epsilon_{n3}^2h^2c_3^2}\,\,ha_{31}L  + \sqrt{1+\epsilon_{n3}h^2c_3^2}\,\,ha_{32} \Biggr(L^2\sqrt{1 + \epsilon_{n2}^2a_{21}^2h^2}\,\,ha_{21} + \frac{L}{\sqrt{1 + \epsilon_{n2}^2a_{21}^2h^2}}\Biggl)\Biggr]|u_n - v_n|. 
\end{align*}
Now,
\begin{align*}
    \phi_n &= 1 + b_1 h L 
    + h b_2 L \left( \sqrt{1 + \epsilon_{n2}^2 a_{21}^2 h^2}\,\, h L a_{21} + \frac{1}{\sqrt{1 + \epsilon_{n2}^2 h^2 a_{21}^2}} \right) \\
    &\quad + b_3 h L \left[ \sqrt{1 + \epsilon_{n3}^2 h^2 c_3^2}\,\, h a_{31} L 
    + \sqrt{1 + \epsilon_{n3}^2 h^2 c_3^2}\,\, h a_{32} \left( L^2 \sqrt{1 + \epsilon_{n2}^2 a_{21}^2 h^2}\,\, h a_{21} 
    + \frac{L}{\sqrt{1 + \epsilon_{n2}^2 a_{21}^2 h^2}} \right) \right], \\
    & = 1 + hL \left( b_1 + \frac{b_2}{\sqrt{1 + \epsilon_{n2}^2h^2a_{21}^2}} + \frac{b_3}{\sqrt{1 
     +\epsilon_{n3}^2h^2c_3^2}}\right) + h^2L^2\left(b_2a_{21}\sqrt{1 + \epsilon_{n2}^2a_{21}^2h^2} +  b_3a_{31}\sqrt{1 + \epsilon_{n3}^2h^2c_3^2} \right. \\
     & \left. + a_{32}b_3\sqrt{1 + \epsilon_{n3}^2h^2c_3^2}\sqrt{1 + \epsilon_{n2}^2a_{21}^2h^2}\right) + b_3L^3h^3a_{21}a_{32}\sqrt{1 + \epsilon_{n3}^2h^2c_3^2}\sqrt{1 + \epsilon_{n2}^2h^2a_{21}^2}, \\
     & \leq 1 + hL + \frac{h^2L^2}{2} - \frac{h^2L^2}{2} \left( b_2a_{21}\sqrt{1 + \epsilon_{n2}^2a_{21}^2h^2} +  b_3a_{31}\sqrt{1 + \epsilon_{n3}^2h^2c_3^2} \right. \\
     & \left. + a_{32}b_3\sqrt{1 + \epsilon_{n3}^2h^2c_3^2} \right) + b_3L^3h^3a_{21}a_{32}\sqrt{1 + \epsilon_{n3}^2h^2c_3^2}\sqrt{1 + \epsilon_{n2}^2h^2a_{21}^2}, \\
      & \leq (1 + \mathcal{C}_3 h) e^{hL}, \\
      & \leq e^{h(1+\mathcal{C}_3)},
\end{align*}
where,
 \begin{align*}
    \mathcal{C}_3 = \sup_{\substack{0 < h \leq b - a \\ n = 0, \dots, N - 1}} 
    \frac{A}{ e^{hL} },
\end{align*}
with 
\begin{align*}
A=  & h^2 L^3 b_3 c_{2} a_{32} \sqrt{1 + \epsilon_{n3}^2 c_3^2 h^2} \sqrt{1 + \epsilon_{n2}^2 c_{2}^2 h^2} \\
&\qquad- \frac{h^2L^2}{2} \biggl( b_2 c_{2} \sqrt{1+\epsilon_{n2}^2 c_{2}^2 h^2}+ b_3 a_{31} \sqrt{1 + \epsilon_{n3}^2 h^2 c_3^2} 
          + b_3 a_{32} \sqrt{1+ \epsilon_{n3}^2 c_3^2 h^2} \biggr).
\end{align*}
Here, $b_3$, $c_2$, $c_3$, $a_{32}$ is constant and $\epsilon_{n2}^2$ and $\epsilon_{n3}^2$ is bounded as it depends on $u$ and $u$ is bounded ensured by Lipschitz continuity. Therefore the term $E_n$ satisfies 
\begin{equation*}
    E_n \leq \sum_{j=0}^{n-1} \varphi_j E_0 + h \sum_{j=1}^{n-1} \left( \sum_{m=j}^{n-1} \varphi_m \right) |\tau_{j-1}| + h |\tau_{n-1}|.
\end{equation*}
Clearly,
\begin{equation*}
    \sum_{m=j}^{n-1} \varphi_m \leq \sum_{m=0}^{n-1} \varphi_m \leq e^{h(L + C_2) n} \leq e^{h(L + C_2) N} = e^{(L + C_2)(b - a)}.
\end{equation*}
With \( E_0 = 0 \) and
\[
\|\tau\|_\infty = \max_{n = 0, \dots, N-1} |\tau_n|,
\]
it can be concluded that for every \( n = 0, 1, \dots, N \),
\begin{equation*}
    E_n \leq h e^{(L + C_2) T} \sum_{j=0}^{n-1} \|\tau\|_\infty = n h e^{(L + C_2)(b - a)} \|\tau\|_\infty \leq (b - a) e^{(L + C_2)(b - a)} \|\tau\|_\infty.
\end{equation*}
Hence,
\begin{equation*}
\lim_{\substack{h \to 0 \\ N h = b - a}} E_n = 0.
\end{equation*}
\end{proof}
\section{Stability Regions} 
Stability regions indicates to the set of values in complex plane where absolute value of the stability polynomial is bounded by one. $y' = \lambda y $ is taken as the test function. Substituting the test function in the expression of the method and replacing  $\lambda h$ by $ z $, $h$ is step length, we obtain ratio of $v_{n+1}$ and $v_n$. The ratio is termed as stability polynomial and it is in the form $v_{n+1} = R(z)v_n $, where $R(z)$ is stability polynomial. Stability region is defined by $|R(z)| \leq 1 $. Bounding the stability polynomial by one ensures that the error will not grow unboundedly. 
Stability polynomials of the two, three, and four stage classical Runge-Kutta methods are as follows
\begin{align}
    R(z) &= 1 + z + \frac{1}{2}z^2, \\
    R(z) &= 1 + z + \frac{1}{2}z^2 + \frac{1}{6}z^3, 
\end{align}
which in general indicate the absolute stability region. The stability polynomial for two stage MQ-RBF Runge-Kutta method (\ref{mq_rk2}) is given by 
\begin{equation*}
    R(z) = 1 + z + \frac{z^2}{2} + \frac{z^3}{6} + \frac{z^4}{9}.
\end{equation*}
The stability polynomial for three stage MQ-RBF Runge-Kutta Method (\ref{eq:mq3}) are obtained as
\begin{enumerate} [label=\Roman*.]
    \item 
    \begin{equation*}
        R(z) = 1 + \frac{z}{10} + \frac{z}{2}\sqrt{1 + \frac{z^2}{3}} + \frac{2}{5}\left[2z -\frac{5}{12}z^2 + \frac{5z^2}{4}(1 + \frac{z}{3}\sqrt{1 + \frac{z^2}{3}} \right],
    \end{equation*}
    \item 
    \begin{equation*}
        R(z) = 1 + \frac{2z}{9} + \frac{z}{2}(1 + \frac{z}{3})\sqrt{1 + \frac{z^2}{3}} + \frac{2}{5}\left[2z + \frac{3}{4}(z^2 + \frac{z^2}{3})\sqrt{1 + \frac{z^2}{3}} \right],
    \end{equation*}
    \item 
    \begin{equation*}
        R(z) = 1 + \frac{z}{6} + \frac{1}{6}(z + z^2)\sqrt{1 + \frac{z^2}{3}} + \frac{4z}{3} + \frac{z^2}{6} + \frac{1}{6}(z +z^2)\sqrt{1 + \frac{z^2}{3}} .
    \end{equation*}
\end{enumerate}
\noindent The stability polynomial of the two stage IMQ-RBF Runge-Kutta method (\ref{eq:imq_2}) is 
\begin{equation*}
    R(z) = 1 + \frac{z}{4} + \frac{1}{2} z^2  \sqrt{(1 - \frac{4 z^2}{9})} + \frac{3}{4}\dfrac{z}{\sqrt{(1 - \frac{4 z^2}{9}})}.
\end{equation*}
The stability polynomial of the three-stage IMQ-RBF Runge-Kutta method \eqref{eq:imq3} are
\begin{enumerate} [label=\Roman*.]
    \item 
    \begin{equation*}
      \begin{aligned}
       R(z) =\; & 1 + \frac{z}{6} 
       + \frac{2}{3}\left( z^2 \sqrt{1-\frac{z^2}{4}} + \frac{z}{\sqrt{1-\frac{z^2}{4}}} \right) \\
       & + \frac{1}{6} \left\{z\sqrt{1+z^2} \left( 2z^2\sqrt{1- \frac{z^2}{4}} - z + \frac{2z}{\sqrt{1 - \frac{z^2}{4}}} \right) + \frac{z}{\sqrt{1+z^2}} \right\},
       \end{aligned}
    \end{equation*}
      \item
      \begin{equation*}
        \begin{aligned}
     R(z) =\; & 1 + \frac{z}{10} 
      + \frac{1}{2}\left( z^2 \sqrt{1-\frac{z^2}{9}} + \frac{z}{\sqrt{1-\frac{z^2}{9}}} \right) \\
      & + \frac{2}{5} \left\{z\sqrt{1+\frac{5}{36}z^2} \left( -\frac{5}{12}z + \frac{5z}{4}\left(z\sqrt{1 - \frac{z^2}{9}} + \frac{1}{\sqrt{1-\frac{z^2}{9}}}\right) \right) + \frac{z}{\sqrt{1+\frac{5}{36}z^2}} \right\},
         \end{aligned}
       \end{equation*}
\item \begin{equation*}
\begin{aligned}
R(z) =\; & 1 + \frac{z}{6} 
+ \frac{1}{6}\left( z^2 \sqrt{1-z^2} + \frac{1}{\sqrt{1-z^2}} \right) \\
& + \frac{2}{3} \left\{z\sqrt{1+\frac{z^2}{4}} \left( \frac{z}{4} + z\sqrt{1 - z^2} + \frac{z}{\sqrt{1-z^2}} \right) + \frac{z}{\sqrt{1+\frac{z^2}{4}}} \right\},
\end{aligned}
\end{equation*}
\item
\begin{equation*}
\begin{aligned}
R(z) =\; & 1 + \frac{2z}{9} 
+ \frac{1}{3}\left( z^2 \sqrt{1-\frac{z^2}{5}} + \frac{z}{\sqrt{1-\frac{z^2}{5}}} \right) \\
& + \frac{4}{9} \left\{z\sqrt{1-\frac{4z^2}{15}} \left( \frac{3z^2}{4} \sqrt{1 - \frac{4z^2}{5}} + \frac{z}{\sqrt{1-\frac{z^2}{5}}} \right) + \frac{z}{\sqrt{1+\frac{4z^2}{15}}}\right\}.
\end{aligned}
\end{equation*}
\end{enumerate}
 Figure\ref{(23)} indicates the stability regions of the proposed methods in comparison to classical RK methods. Figure\ref{(23)} (a) illustrates stability regions of two stage Runge-Kutta and IMQ-Runge Kutta method. Note that both the stability regions are in the negative real axis. Figure \ref{(23)}(b) shows stability region for RK-2 and MQ-RK-2 method. Figure\ref{(23)} (c) shows the stability region for various IMQ-RK-3 and RK-3 methods. The stability region for the MQ-RK-3 and RK-3 methods is shown in graph Figure \ref{(23)}(d). In each case, it is evident that the stability region for the classical Runge-Kutta method is larger than all enhanced Runge-Kutta methods.
\begin{figure}
    \centering
     \begin{subfigure}[b]{0.48\textwidth}
         \centering
         \includegraphics[trim={11cm 0 0 3 cm},clip, width=12cm, height=7cm]{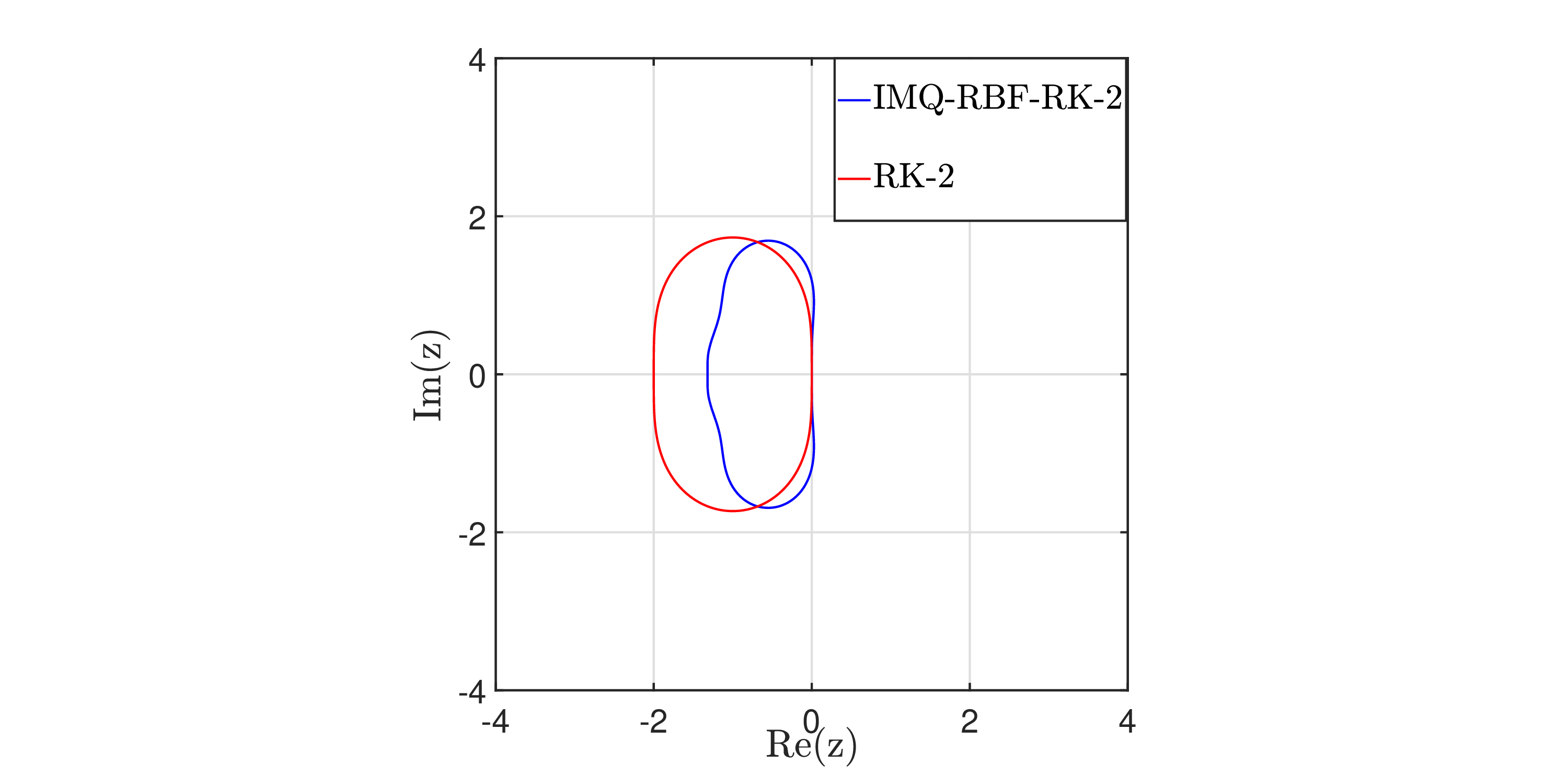}
         \caption{Stability regions of RK2 and IMQ-RK2.}
         \label{(22a))}
     \end{subfigure}
     \hfill
     \begin{subfigure}[b]{0.48\textwidth}
         \centering
         \includegraphics[trim={8cm 0cm 0 0},clip, width=12cm, height=7cm]{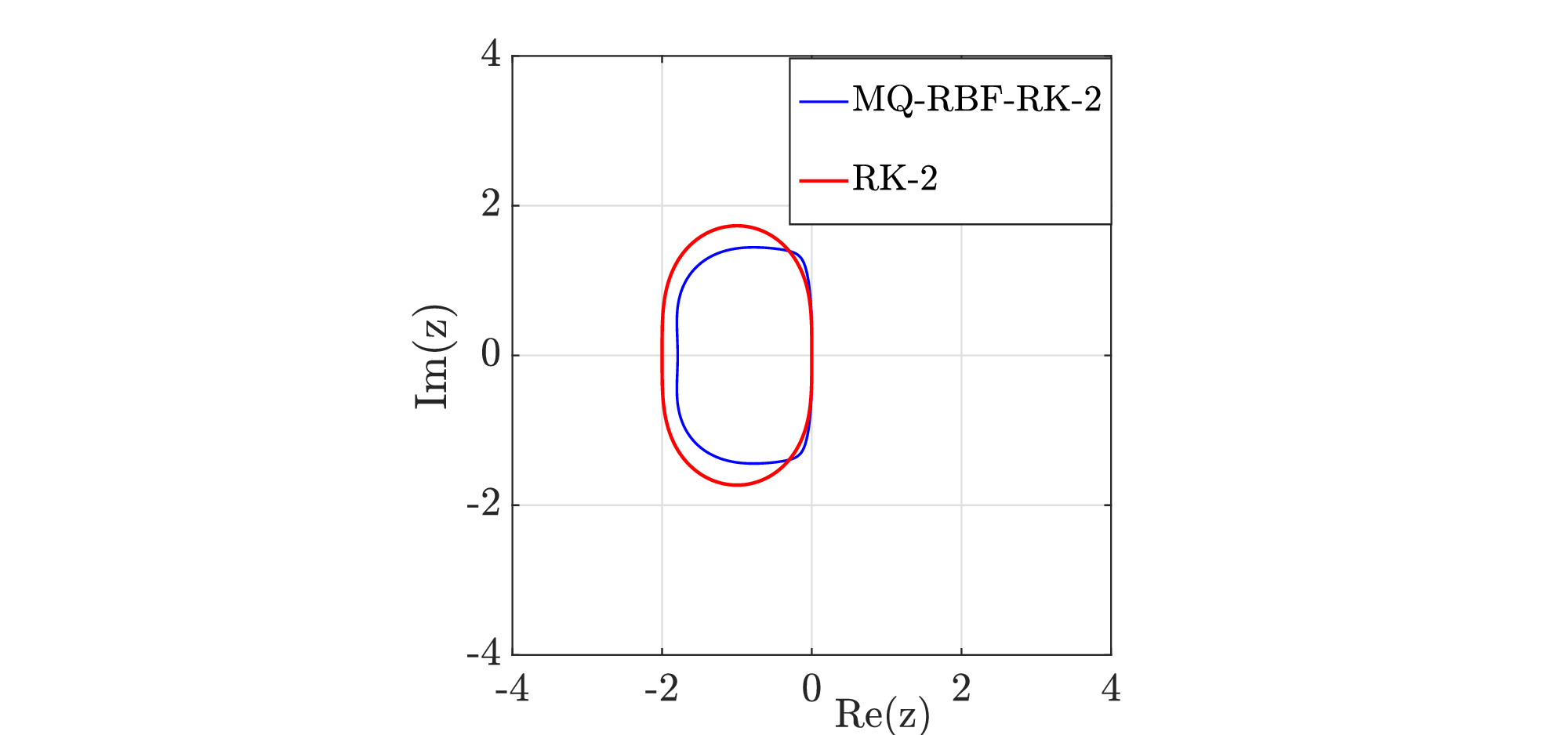}
         \caption{Stability regions of RK2  and MQ-RK2.}
         \label{(22b))}
     \end{subfigure}
     \medskip
    \centering
     \begin{subfigure}[b]{0.48\textwidth}
         \centering
         \includegraphics[trim={10cm 0 0 0 cm},clip, width=11.7cm, height=7cm]{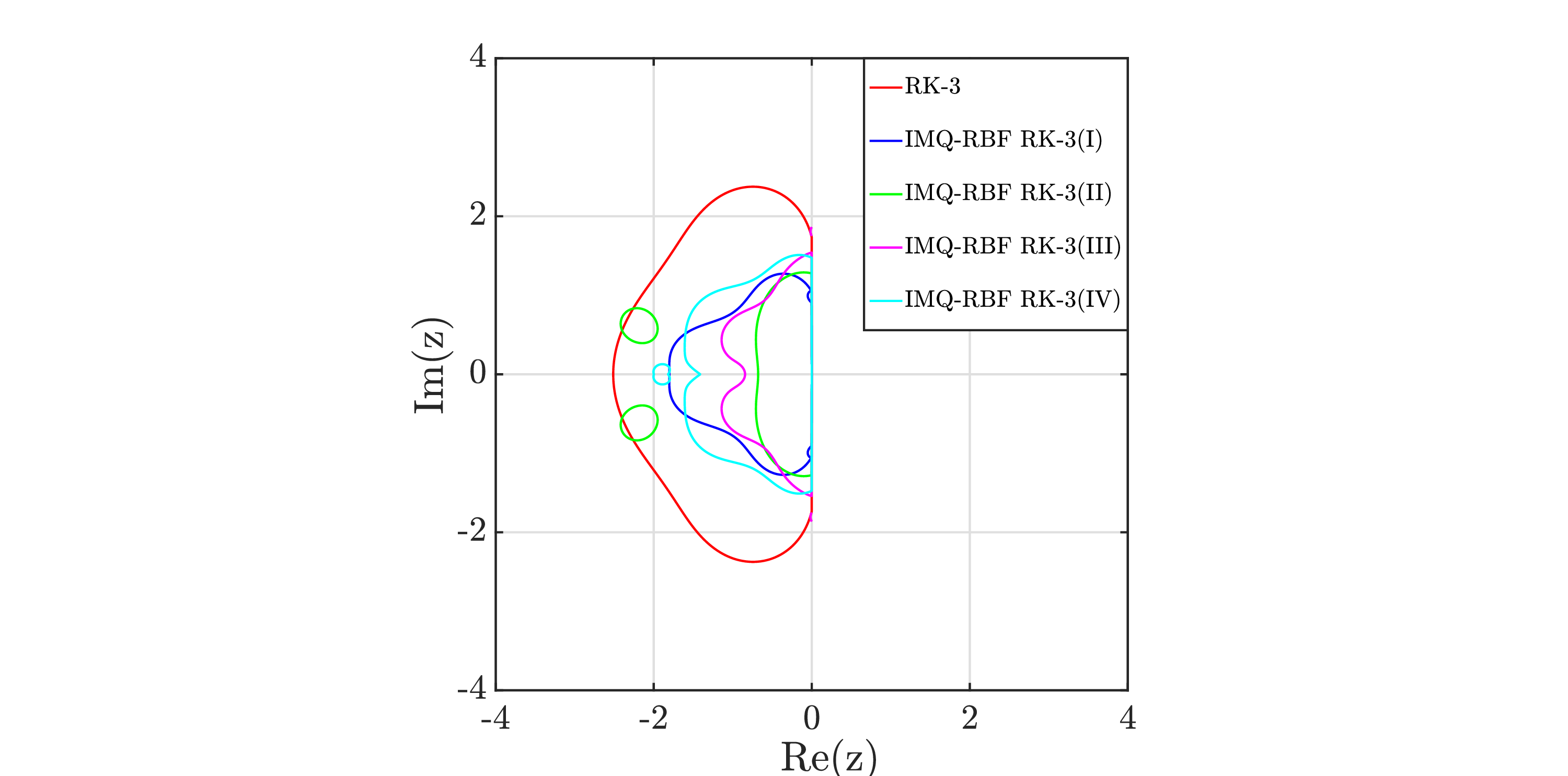}
         \caption{Stability regions of RK3  and IMQ-RK3.}
         \label{(23a)}
     \end{subfigure}
     \hfill
     \begin{subfigure}[b]{0.48\textwidth}
         \centering
         \includegraphics[trim={12cm 0 0 0},clip, width=11.8cm, height=7cm]{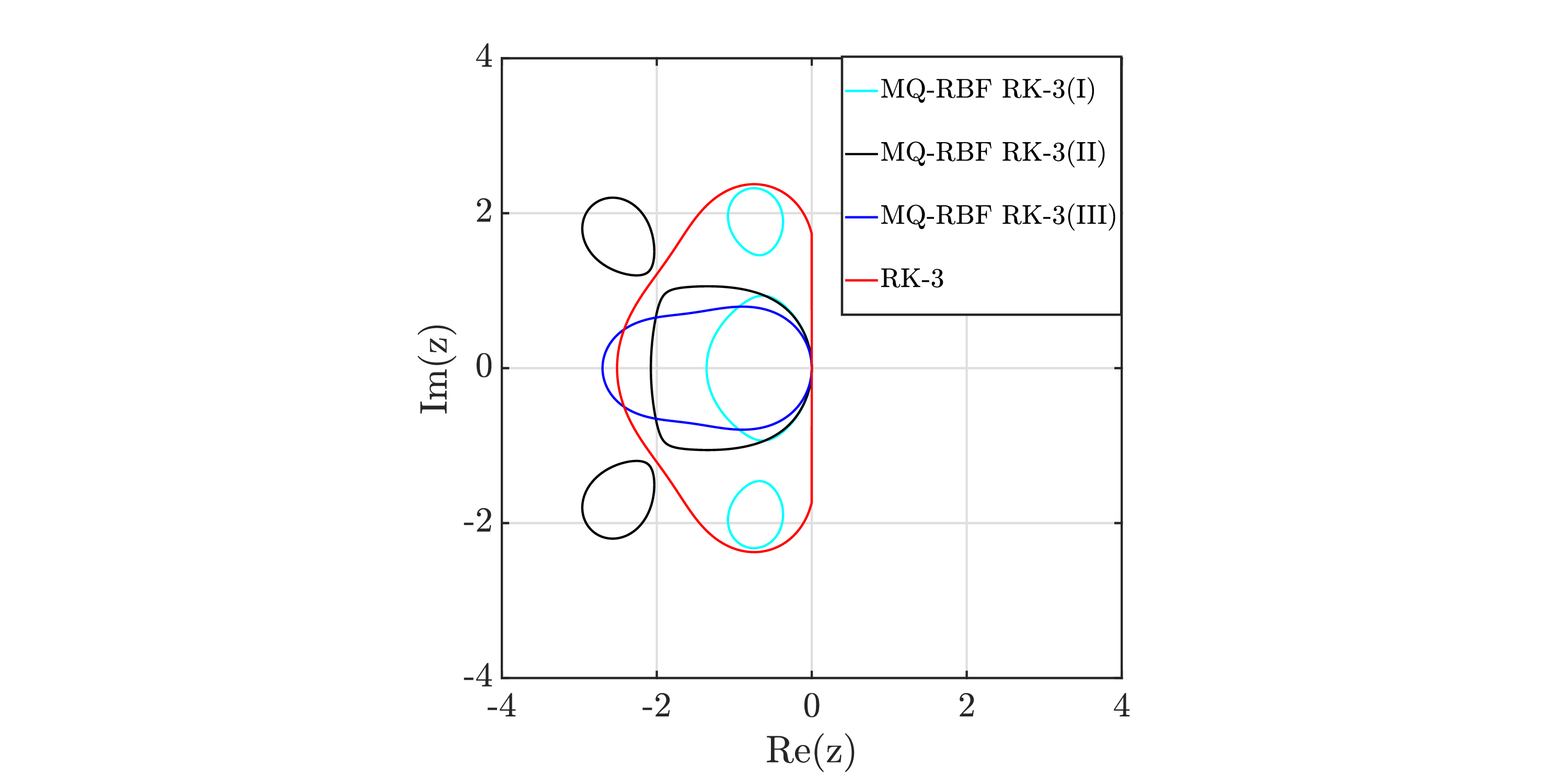}
         \caption{Stability regions of RK3  and MQ-RK3.}
         \label{(23b)}
     \end{subfigure}
     \caption{Stability region for various classical, MQ and IMQ -RBF Runge Kutta schemes}
     \label{(23)}
\end{figure}

\section{Numerical Results}
  In this section, numerical examples of MQ and IMQ RBF Runge-Kutta methods are presented and discussed in support of the theory presented in previous sections. All calculations are performed with MATLAB R2022A. A comparison is followed with the corresponding classical Runge-Kutta method. All the errors preseneted here are calculated using $L_1-$norm.
\section*{Example 6.1}
Consider the ordinary differential equation,
\begin{equation*}
    \frac{du}{dt} = -u^2, \quad 0 < t \leq 1,
\end{equation*}
with the initial condition;
\begin{equation*}
    u(0) = 1.
\end{equation*}
The exact solution of the above ordinary differential equation is $u(t) =\frac{1}{t+1}$, decreasing in nature in $0 < t\leq1$. The global error evaluated at 
$t=1$ is reported for the standard RK2 method and its enhanced variants, namely MQ-RK2 and IMQ-RK2 in Table \ref{tab:error_order_two}. Table 2 demonstrates that various RK3 methods achieve third-order accuracy, while different IMQ-RK3 methods attain fourth-order accuracy. Table 4 shows that the various MQ-RK3 methods described in Section 3 achieve fourth-order accuracy.
\begin{table}[htbp!]
    \centering
     \caption{Global error and order of accuracy by example 6.1 by RK2, MQ, IMQ-RK2.}
    \begin{tabular}{c c c c c c c}
        \toprule
        N & \multicolumn{2}{c}{RK2} & \multicolumn{2}{c}{MQ-RK2} & \multicolumn{2}{c}{IMQ-RK2} \\
        \cmidrule(lr){2-3} \cmidrule(lr){4-5} \cmidrule(lr){6-7}
          & Error & Order & Error & Order & Error & Order\\
        \midrule
        10  & 1.119140e-03   & --   & 9.316803e-06 & -- & 1.594597e-04  & --   \\
        20  & 2.628612e-04  & 2.0900 & 1.487789e-06  & 2.6467 & 1.763600e-05  & 3.1766 \\
        40  &  6.368993e-05  & 2.0452 & 2.026835e-07 & 2.8759 & 2.074312e-06 & 3.0878 \\
        80  & 1.567527e-05  & 2.0226 & 2.626486e-08 &2.9480& 2.516187e-07  & 3.0433 \\
        160 & 3.888293e-06 & 2.0113 & 3.338011e-09  & 2.9761  &3.098107e-08    & 3.0218 \\
        320 & 9.682818e-07 & 2.0056 &  4.205879e-10   & 2.9885 & 4.205879e-10 & 3.0109 \\
        \bottomrule
    \end{tabular}
    \label{tab:error_order_two}
\end{table}
\begin{table}[htbp!]
\centering
\caption{Global error and order of accuracy for example 6.1 by RK3 and IMQ-RK3.}
\begin{tabular}{c c c c c c c c c}
\toprule
 N & \multicolumn{2}{c}{RK3 I} & \multicolumn{2}{c}{IMQ-RK3 I} & \multicolumn{2}{c}{RK3 II} & \multicolumn{2}{c}{IMQ-RK3 II} \\
\cmidrule(lr){2-3} \cmidrule(lr){4-5} \cmidrule(lr){6-7}  \cmidrule(lr){8-9}
& Error & Order & Error & Order & Error & Order & Error & Order \\
\midrule
10  & 2.642520e-05 &       & 4.633848e-06 &       & 4.783447e-05 &       & 1.592061e-06 &       \\
20  & 2.916366e-06 & 3.1797 & 2.573850e-07 & 4.1692 & 5.580618e-06 & 3.0999 & 9.390044e-08 & 4.0836 \\
40  & 3.439349e-07 & 3.0840 & 1.509396e-08 & 4.0927 & 6.738067e-07 & 3.0246 & 5.681100e-09 & 4.0469 \\
80  & 4.181551e-08 & 3.0400 & 9.130775e-10 & 4.0235 & 8.280145e-08 & 3.0152 & 3.492980e-10 & 4.0236 \\
160 & 5.155666e-09 & 3.0185 & 5.614342e-11 & 4.0235 & 1.026114e-08 & 3.0125 & 2.165412e-11 & 4.0117 \\
320 & 6.400902e-10 & 3.0098 & 3.480549e-12 & 4.0117 & 1.277089e-09 & 3.0063 & 1.347748e-12 & 4.0063 \\
\midrule
N & \multicolumn{2}{c}{RK3 III} & \multicolumn{2}{c}{IMQ-RK3 III} & \multicolumn{2}{c}{RK3 IV} & \multicolumn{2}{c}{IMQ-RK3 IV} \\
\cmidrule(lr){2-3} \cmidrule(lr){4-5} \cmidrule(lr){6-7} \cmidrule(lr){8-9}
& Error & Order & Error & Order & Error & Order & Error & Order \\
\midrule
10  & 4.716981e-05 &       & 5.617946e-06 &       & 4.786103e-05 &       & 2.274155e-06 &       \\
20  & 5.548240e-06 & 3.0878 & 3.166580e-07 & 4.1490 & 5.583189e-06 & 3.0997 & 1.311771e-07 & 4.1157 \\
40  & 6.720259e-07 & 3.0454 & 1.870708e-08 & 4.0813 & 6.739872e-07 & 3.0503 & 7.847559e-09 & 4.0631 \\
80  & 8.270012e-08 & 3.0226 & 1.136403e-09 & 4.0140 & 8.281589e-08 & 3.0247 & 4.798086e-10 & 4.0317 \\
160 & 1.025485e-08 & 3.0116 & 7.002599e-11 & 4.0204 & 1.026189e-08 & 3.0166 & 2.966272e-11 & 4.0157 \\
320 & 1.276703e-09 & 3.0058 & 4.344858e-12 & 4.0105 & 1.277132e-09 & 3.0063 & 1.844919e-12 & 4.0076 \\
\bottomrule
\end{tabular}
\end{table}
\begin{table}[htbp!]
    \centering
     \caption{Global error and order of accuracy for example 6.1 by various MQ-RK3 methods.}
    \begin{tabular}{c c c c c c c}
        \toprule
        N & \multicolumn{2}{c}{MQ-RK3 I}  & \multicolumn{2}{c}{MQ-RK3 II} & \multicolumn{2}{c}{MQ-RK3 III} \\
        \cmidrule(lr){2-3} \cmidrule(lr){4-5}  \cmidrule(lr){6-7}
          & Error & Order & Error & Order & Error & Order \\
        \midrule
        10  &  4.536178e-06    & -- & 1.609152e-06  & -- & 3.020182e-06       & --    \\
        20  & 2.720225e-07    & 4.0597 &  9.790250e-08  &  4.0388 & 1.799044e-07 &  4.0693 \\
        40  & 1.655922e-08   &   4.0380&  6.007887e-09  & 4.0264 & 1.092505e-08     &  4.0415 \\
        80  &1.020717e-09    &   4.0200 &  3.718700e-10 & 4.0140 &  6.727541e-10   &   4.0214 \\
        160 &  6.334999e-11  &   4.0101& 2.312794e-11  &  4.0071 &  4.173539e-11       & 4.0107\\
        320 & 3.945511e-12 & 4.0051 &1.441736e-12  &    4.0038 &  2.599143e-12  &  4.0052\\
        \bottomrule
    \end{tabular}
\end{table}
\section*{Example 6.2}
Next, we consider the ordinary differential equation  as
\begin{equation*}
    \frac{du}{dt} = -4t^3 u^2, \quad -10 < t \leq 0, \quad u(-10) = \frac{1}{10001}.
\end{equation*}
This is a stiff problem with the exact solution $u(t) = \frac{1}{t^4+1}$
where the solution exhibits rapid variation over the interval $[-10, 0]$. The global error is evaluated at $t = 0$. Table 4 shows that the numerical solution achieves second-order accuracy with the standard RK2 method and third-order accuracy with the enhanced variants MQ-RK2 and IMQ-RK2. As shown in Table 5, the various RK3 and IMQ-RK3 methods yield third and fourth order accuracy, respectively. These RK3 and IMQ-RK3 methods are discussed in Sections 2 and 3, respectively. Table 6 presents the global error and confirms fourth-order accuracy for the ODE using various MQ-RBF methods described in Section 3.
\begin{table}[htbp!]
    \centering
    \caption{Global error and order of accuracy for example 6.2 by RK2, MQ-RK2 and IMQ-RK2}
    \begin{tabular}{c c c c c c c}
        \toprule
        N & \multicolumn{2}{c}{RK2} & \multicolumn{2}{c}{MQ-RK2} & \multicolumn{2}{c}{IMQ-RK2} \\
        \cmidrule(lr){2-3} \cmidrule(lr){4-5} \cmidrule(lr){6-7}
          & Error & Order & Error & Order & Error & Order \\
        \midrule
        200  &  9.422354e-02    & -- & 3.148990e-02  & --  & 9.422354e-02   & --    \\
        400  & 1.344393e-02  & 2.0900 & 4.055329e-03 & 2.9570 & 1.344393e-02  & 2.8091 \\
        800  &  1.738854e-03  & 2.0452 & 5.197901e-04 &2.9638
 & 1.738854e-03 & 2.9507\\
        1600  & 2.201011e-04  & 2.0226 & 6.588221e-05 & 2.9800 & 2.201011e-04   & 2.9819\\
        3200 &  2.766968e-05 & 2.0113 &8.294091e-06  & 2.9897 & 2.766968e-05    & 2.9918 \\
        6400 &3.468332e-06 & 2.0056 & 1.040483e-06 &2.9948 & 3.468332e-06 & 2.9960\\
        \bottomrule
    \end{tabular}
\end{table}
\begin{table}[htbp!]
    \centering
     \caption{Global error and order of accuracy for example 6.2 by RK3, IMQ-RK3.}
    \begin{tabular}{c c c c c c c c c}
        \toprule
        N & \multicolumn{2}{c}{RK3 I} & \multicolumn{2}{c}{IMQ-RK3 I} & \multicolumn{2}{c}{RK3 II} & \multicolumn{2}{c}{IMQ-RK3 II}\\
        \cmidrule(lr){2-3} \cmidrule(lr){4-5} \cmidrule(lr){6-7} \cmidrule(lr){8-9}
          & Error & Order & Error & Order & Error & Order & Error & Order\\
        \midrule
        200 & 4.341821e-02   &--   & 1.599661e-03 & -- & 4.344246e-02    & --& 4.063955e-04  & --  \\
        400 &  5.854329e-03   & 2.8907   & 1.047616e-04 & 3.9326 &  5.826218e-03   &  2.8985& 2.661830e-05 &  3.9324\\
        800 &  7.493172e-04   &  2.9659   & 6.705566e-06 & 3.9656 & 7.433671e-04 & 2.9704 & 1.703303e-06 & 3.9660 \\
       1600 &  9.460367e-05    & 2.9856  & 4.241174e-07 & 3.9828 & 9.369507e-05 & 2.9880&  1.077257e-07   & 3.9829\\
       3200 &  1.188195e-05  &  2.9931  & 2.666491e-08 &  3.9914 & 1.175767e-05  & 2.9944&6.777914e-09  &  3.9904 \\
       6400 &  1.488713e-06    & 2.9966  &  1.690263e-09 & 3.9796 & 1.472522e-06  &  2.9972& 4.157718e-10  & 4.0270 \\
       \midrule
        N & \multicolumn{2}{c}{RK3 III} & \multicolumn{2}{c}{IMQ-RK3 III} & \multicolumn{2}{c}{RK3 IV} & \multicolumn{2}{c}{IMQ-RK3 IV}\\
        \cmidrule(lr){2-3} \cmidrule(lr){4-5} \cmidrule(lr){6-7} \cmidrule(lr){8-9}
          & Error & Order & Error & Order & Error & Order  & Error & Order \\
        \midrule
200 & 6.749350e-02 &-- &  2.373662e-03  & -- & 4.768849e-02   &-- & 7.207373e-04  &--\\
400 & 9.267981e-03 &   2.8644  &   1.666272e-04 & 3.8324 &6.425003e-03  & 2.8919 & 4.757574e-05 & 3.9212 \\
800 & 1.186675e-03  &  2.9653  &  1.068826e-05  &  3.9625 &8.204686e-04 & 2.9692 &3.055810e-06  & 3.9606   \\
1600 &1.496573e-04    &  2.9872   &  4.660953e-07 & 4.5193 &1.034367e-04 & 2.9877 & 1.936161e-07  & 3.9803   \\
3200 &  1.878283e-05  &  2.9942   & 2.926613e-08 & 3.9933 &1.298126e-05 & 2.9942& 1.219565e-08   &  3.9888   \\
6400 & 2.352475e-06  &  2.9972  & 1.852862e-09  & 3.9814 & 1.625848e-06 &  2.9972& 7.586857e-10  &   4.0067    \\
\bottomrule
    \end{tabular}
\end{table}

\begin{table}[htbp!]
    \centering
    \caption{Global error and order of accuracy for example 6.2 by various MQ-RK3 methods.}
    \begin{tabular}{c c c c c c c}
        \toprule
        N & \multicolumn{2}{c}{MQ-RK3 I}  & \multicolumn{2}{c}{MQ-RK3 II} & \multicolumn{2}{c}{MQ-RK3 III} \\
        \cmidrule(lr){2-3} \cmidrule(lr){4-5}  \cmidrule(lr){6-7}
          & Error & Order & Error & Order & Error & Order \\
        \midrule
        200  & 4.647968e-04    & -- &4.766731e-04 & -- & 1.093432e-04    & --    \\
        400  &2.913064e-05    &3.9960 & 3.310703e-05 &  3.8478 &  5.926462e-06  &  4.2055 \\
        800  &1.821534e-06    &   3.9993 & 2.212692e-06   & 3.9033 &  3.397134e-07 &  4.1248 \\
        1600  & 1.138538e-07      &   3.9999 &  1.411230e-07  & 3.9708&  2.024166e-08     &   4.0689 \\
        3200 &  7.122291e-09 &  3.9987& 8.862087e-09&  3.9932 &  1.230801e-09   & 4.0397 \\
        6400 & 4.345316e-10   &  4.0348 & 5.673884e-10     &  3.9652&  8.104606e-11  & 3.9247\\
        \bottomrule
    \end{tabular}
\end{table}
\section*{Example 6.3}
Consider
\begin{equation*}
    \frac{du}{dt} =  \frac{2t^2 - u}{t^2u-t}, \quad 1 < t \leq 2, \quad
    u(1) = 2.
\end{equation*}
The non-separable ordinary differential equation has the exact solution 
$u(t) = \frac{1}{t} + \sqrt{\frac{1}{t}^2 + 4t -4},$ which is increasing over the interval $t = 1$ to $t = 2$. The global error at the final time $t = 2$ is computed using RK2, MQ-RK2, and IMQ-RK2 methods, as shown in Table 7. The order of accuracy is also calculated for these methods. RK2 achieves second-order accuracy, while MQ-RK2 and IMQ-RK2 attain third-order accuracy. Table 8 presents the global errors for various RK3 and IMQ-RK3 methods. The different IMQ-RK3 variants discussed in Section 3 outperform the standard RK3 methods in terms of both global error and accuracy. Table 9 shows the solution of the above ODE using various MQ-RK3 methods, demonstrating fourth-order accuracy.

\begin{table}[htbp!]
    \centering
     \caption{Global error and order of accuracy for example 6.3 by RK2, MQ-RK2 and IMQ-RK2.}
    \begin{tabular}{c c c c c c c}
        \toprule
        N & \multicolumn{2}{c}{RK2}  & \multicolumn{2}{c}{MQ-RK2} & \multicolumn{2}{c}{IMQ-RK2} \\
        \cmidrule(lr){2-3} \cmidrule(lr){4-5}  \cmidrule(lr){6-7}
          & Error & Order & Error & Order & Error & Order \\
        \midrule
        10  & 8.789865e-04   & -- &2.184352e-04 & -- & 2.106559e-04    & --    \\
        20  & 2.076179e-04   & 2.0819 & 2.542775e-05 & 3.1027 & 2.386215e-05  &  3.1421 \\
        40  &5.048556e-05    &  2.0400 & 3.064144e-06 &3.0528 & 2.836513e-06  & 3.0725 \\
        80  & 1.244991e-05   & 2.0197 & 3.757647e-07 & 3.0276 & 3.460363e-07  & 3.0351 \\
        160 & 3.091751e-06  & 2.0096 & 4.652147e-08&  3.0139 & 4.272868e-08    & 3.0176 \\
        320 & 7.703476e-07 & 2.0048 &5.787081e-09  &  3.0070 & 5.308747e-09  & 3.0088\\
        \bottomrule
    \end{tabular}
\end{table}
\begin{table}[h]
    \centering
    \caption{Global error and order of accuracy for example 6.3 by RK3 and IMQ-RK3.}
    \begin{tabular}{c c c c c c c c c}
        \toprule
        N & \multicolumn{2}{c}{RK3 I} & \multicolumn{2}{c}{IMQ-RK3 I} & \multicolumn{2}{c}{RK3 II} & \multicolumn{2}{c}{IMQ-RK3 II}\\
        \cmidrule(lr){2-3} \cmidrule(lr){4-5} \cmidrule(lr){6-7} \cmidrule(lr){8-9}
          & Error & Order & Error & Order & Error & Order & Error & Order\\
        \midrule
        10 & 6.819460e-06 & --& 9.102354e-07 &-- & 5.447556e-05 &-- &1.036981e-05  &-- \\
        20 &  1.559008e-06  & 2.1290 &7.556264e-08  &3.5905 & 6.523486e-06  & 3.0619 & 1.007890e-06 & 3.3630 \\
        40  & 2.280679e-07 & 2.7731 & 5.461295e-09   & 3.7904&7.935154e-07 & 3.0393 &  6.019957e-08  &  4.0654 \\
        80  & 3.019340e-08 &  2.9172 &3.659935e-10  &3.8994  &9.775898e-08 & 3.0210 & 7.278295e-08   &  -0.2738\\
        160 &  3.869651e-09 & 2.9640 &  2.366685e-11 & 3.9509 &1.212638e-08  &  3.0111 & 2.518786e-09 & 4.8528 \\
        320 & 4.891940e-10 & 2.9837 & 1.504130e-12 & 3.9759 &1.509893e-09 &3.0056 & 1.009757e-10 & 4.6406\\
       \midrule
        N & \multicolumn{2}{c}{RK3 III} & \multicolumn{2}{c}{IMQ-RK3 III} & \multicolumn{2}{c}{RK3 IV} & \multicolumn{2}{c}{IMQ-RK3 IV}\\
        \cmidrule(lr){2-3} \cmidrule(lr){4-5} \cmidrule(lr){6-7} \cmidrule(lr){8-9}
          & Error & Order & Error & Order & Error & Order  & Error & Order \\
        \midrule  
 10 & 6.916137e-05& --&3.238976e-05 &-- & 7.024289e-05  &-- & 9.102354e-07   &-- \\
        20 & 8.770902e-06 & 2.9792& 2.286074e-06 & 3.8246&8.622116e-06 & 3.0262&   7.556264e-08   & 3.5905\\
        40  & 1.099991e-06 & 2.9952&1.427746e-07 &4.0011 &1.063681e-06   &3.0190 &  5.461295e-09 & 3.7904   \\
        80  & 1.376718e-07 & 2.9982&8.792510e-09 &  4.0213 & 1.320313e-07   &3.0101 &  3.659935e-10& 3.8994 \\
        160 & 1.721441e-08 &2.9995 &5.430785e-10   &  4.0170 &1.644235e-08   &3.0054 & 2.366685e-11 & 3.9509 \\
        320 &  2.152050e-09 &2.9998 &3.370415e-11   &4.0102  & 2.051421e-09& 3.0027 &  1.504130e-12   &3.9759 \\
\bottomrule
    \end{tabular}
\end{table}
\begin{table}[htbp!]
    \centering
    \caption{Global error and order of accuracy for example 6.3 by various MQ-RK3 methods.}
    \begin{tabular}{c c c c c c c}
        \toprule
        N & \multicolumn{2}{c}{MQ-RK3 I}  & \multicolumn{2}{c}{MQ-RK3 II} & \multicolumn{2}{c}{MQ-RK3 III} \\
        \cmidrule(lr){2-3} \cmidrule(lr){4-5}  \cmidrule(lr){6-7}
          & Error & Order & Error & Order & Error & Order \\
        \midrule
        10  & 8.158208e-06   & -- &1.795212e-04 & -- & 2.711265e-05       & --    \\
        20  & 4.606486e-07    & 4.1465 &  4.752717e-06  & 5.2393&  1.480904e-06   &  4.1944 \\
        40  &2.749617e-08    &   4.0664& 1.419743e-07  & 5.0651 & 8.654651e-08   &  4.0969 \\
        80  &1.681109e-09    &   4.0317 & 5.324079e-09 &  4.7370 &  5.230366e-09  & 4.0485 \\
        160 & 1.039404e-10  &  4.0156& 2.414269e-10&   4.4629 & 3.214189e-10     & 4.0244\\
        320 &  6.463274e-12 & 4.0073 &1.257616e-11  &  4.2628 & 1.991918e-11  &  4.0122\\
        \bottomrule
    \end{tabular}
\end{table}
\section*{Example 6.4}
Explicit RBF-Runge method for scalar is extended to system case by using component wise operation. We consider the system of equations  
\begin{equation*}
\begin{aligned}
 \frac{du_1}{dt}  & =  e^t - 5u_1 + 3u_2,\\
 \frac{du_2}{dt}  & = -3u_1 + u_2,
\end{aligned}
\end{equation*}
with initial data
\begin{equation*}
     u_1(0) = 1, u_2(0) = 0.
\end{equation*}
The exact solution of the above ordinary differential equation $u
_1(t) = (1-2t)e^2t,\,\, u_2(t) = (\frac{1}{3} - 2t)e^{-2t} - \frac{1}{3}e^t$. Solution is approximated at final time $t = 5$ and shown in table 10. It is observed that truncation error appeared in RK-2 method follows second order accuracy, whereas MQ-RK2 and IMQ-RK2 attains third order accuracy. Furthermore, the error decreases as number of grid points increases.

\begin{table}[h]
    \centering
    \caption{Global error and order of accuracy for example 6.4 by RK-2, MQ-RK2, IMQ-RK2.}
    \begin{tabular}{c c c c c c c }
        \toprule
        $u_1$ & \multicolumn{2}{c}{RK 2} & \multicolumn{2}{c}{IMQ-RK2} & \multicolumn{2}{c}{MQ-RK2 } \\
        \cmidrule(lr){2-3} \cmidrule(lr){4-5} \cmidrule(lr){6-7} 
         N & Error & Order & Error & Order & Error & Order \\
        \midrule
        20 &  3.8578e-01   & --& 3.0277e-02   &-- & 4.0712e-02 &--  \\
        40 &  6.6058e-02      &   2.5460 &6.3264e-03   &2.2588 & 3.2458e-03   & 3.6488  \\
        80  &  1.3425e-02 &  2.2989 & 7.1610e-04   & 3.1431&3.4340e-04 & 3.2406 \\
        160  & 3.0077e-03 &  2.1582 & 8.2946e-05  &3.0887  &4.0364e-05 & 3.0887	 \\
        320 &   7.1054e-04 &  2.0817&  9.5653e-06 & 3.1163 &5.2049e-06  &  2.9551 \\
        640 &  1.7259e-04  &   2.0415 & 1.0391e-06  & 3.2024 &7.7303e-07 &2.7513 \\
       \midrule
        $u_2$ & \multicolumn{2}{c}{RK 2} & \multicolumn{2}{c}{IMQ-RK2} & \multicolumn{2}{c}{MQ-RK2} \\
        \cmidrule(lr){2-3} \cmidrule(lr){4-5} \cmidrule(lr){6-7} 
         N & Error & Order & Error & Order & Error & Order  \\
        \midrule  
 20 & 2.6347e-02 & --& 2.8261e-02 &-- & 3.6335e-03   &--  \\
        40 & 2.7899e-02 & -0.0826	 &  2.9749e-03   & 3.2479&1.8151e-03  & 1.0013\\
        80  &  9.0268e-03 & 1.6279 &4.6315e-04   &2.6833 &3.6812e-05  &2.7065  \\
        160  & 2.4810e-03 &1.8633&6.4613e-05  &  2.8416 & 2.7808e-04    &2.9173 \\
        320 & 6.4646e-04 &1.9403 &8.8737e-06  &  2.8642 &4.3812e-06   &3.0708  \\
        640 &   1.6478e-04  & 1.9720 &1.2748e-06     &2.7992  & 4.1112e-07& 3.4137 \\
\bottomrule
    \end{tabular}
\end{table}
\section*{Example 6.5} The Duffing equation \cite{mei2017symplectic} is given by 
\begin{align*}
    & q'' + \omega^2q = k^2\left( 2q^3 - q\right), 0 \le t \leq 20,\\
    & q(0) = 0, q'(0) = \omega, 
\end{align*}
 with $0 \leq k < \omega$. Choosing $p = q'$ and $u = ( p, q)^T$, the above-mentioned second-order differential equation reduces to a first-order differential equation system in the form $u' = f(u)$, where $f(u) = (-\omega^2 q + k^2(2q^3 - q),\, p)^T$. The analytic solution is given by $q(t) = sn(\omega, k/ \omega).$ Here, $ \omega = 10$, $k = 0.03$ and $h = 0.05$ is taken in our experimental results. From figure \ref{duff}, it is clearly visible that the exact solution blow up over time whereas, the approximate solution computed by MQ-RK2 and IMQ-RK2 method damp over time. Duffing equation can be written as Hamiltonian system, so it is structure preserving but RK2, MQ-RK2, IMQ-RK2 are not so. So, while approximating the Duffing equation it is either blowing up or damping. In our future work, we will explore the sympletic Runge-Kutta method.
\begin{figure}[H]
  \centering
  \includegraphics[trim={1cm 0 1cm 0}, clip, width=9cm, height=6cm]{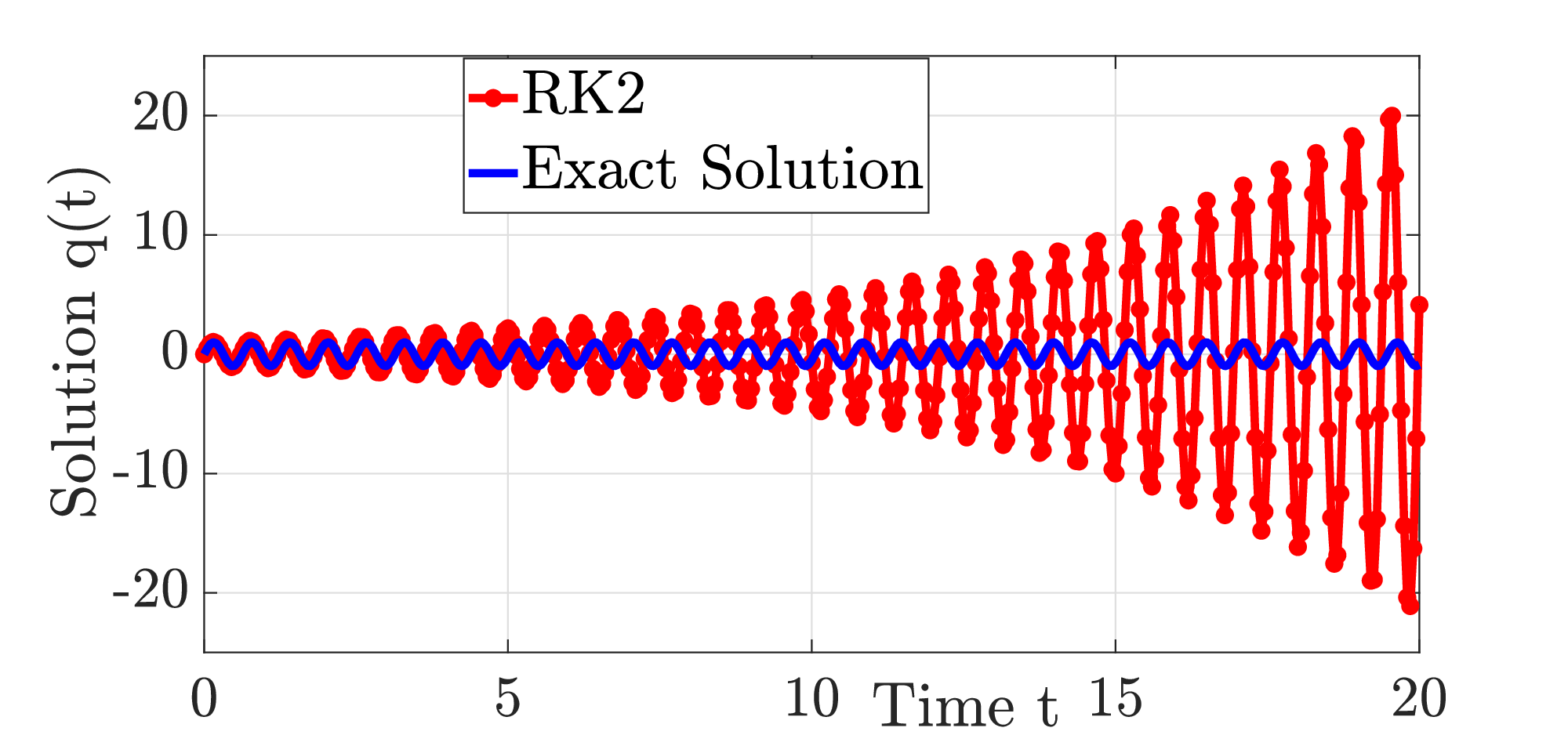}\\
  \includegraphics[trim={0 0 1.5cm 0.5cm}, clip, width=8.35cm, height=5.5cm]{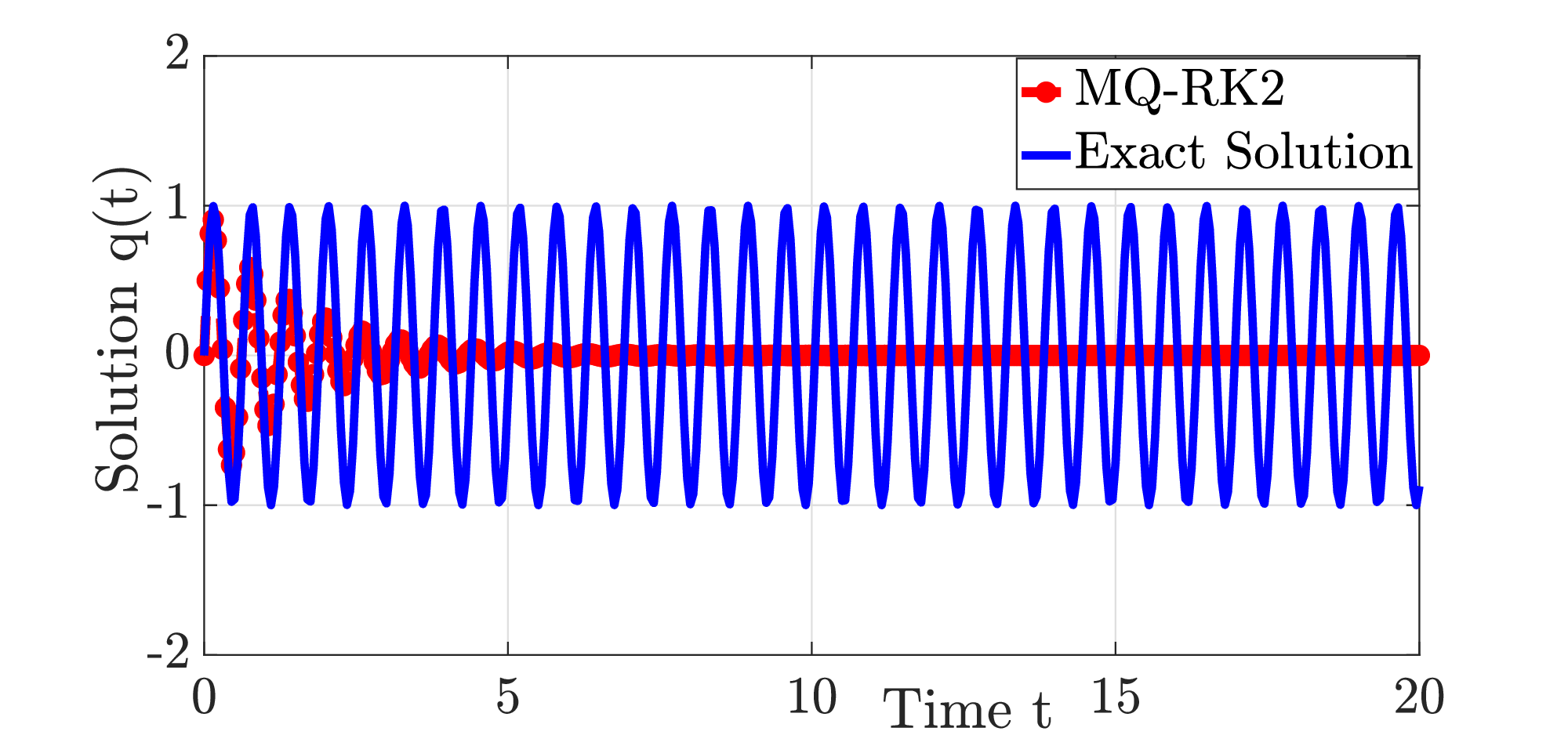}
  \includegraphics[trim={2cm 0 2.5cm 0}, clip, width=8.5cm, height=5.7cm]{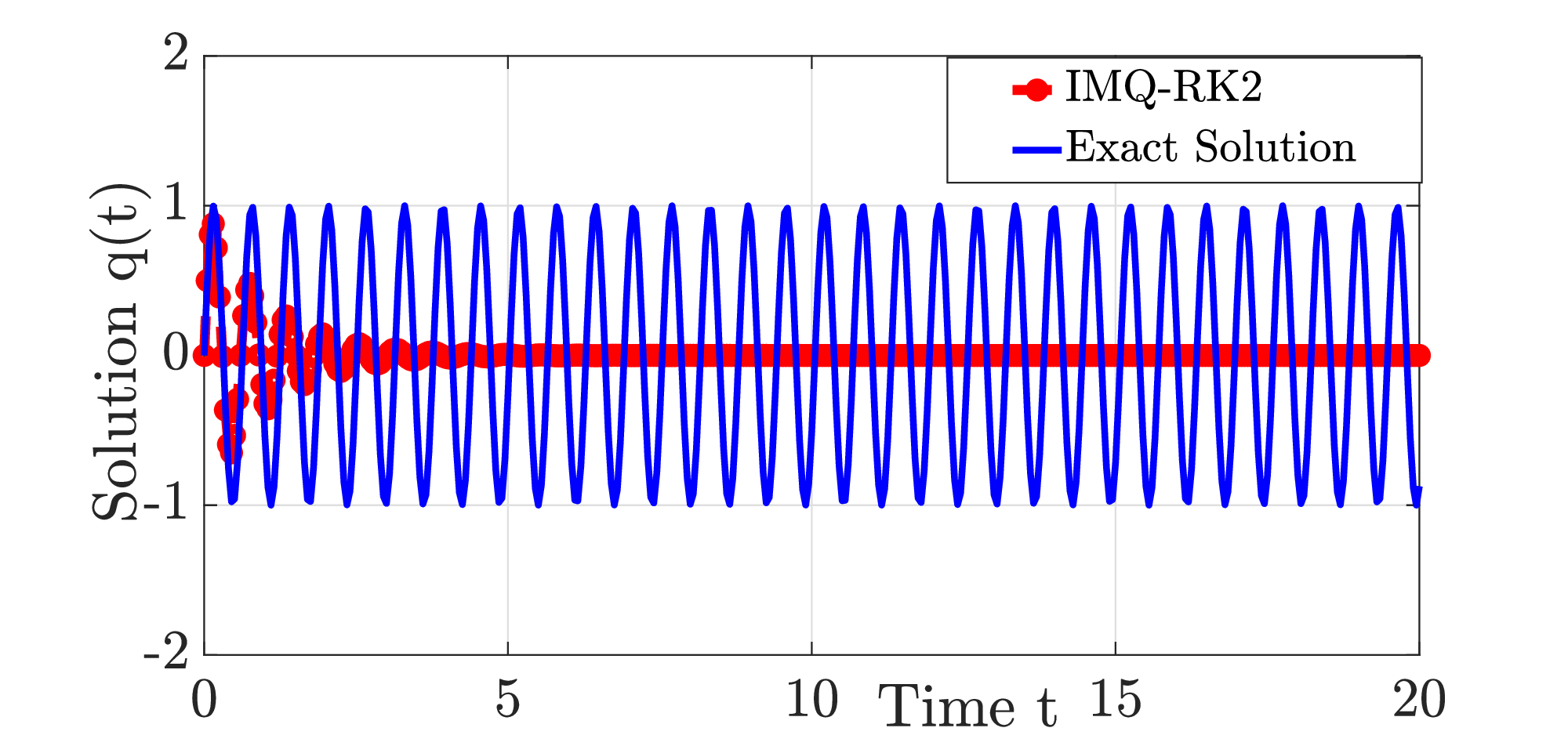}
\caption{The three graphs represents solution of example 5 where approximate solution by RK2, MQ-RK2 and IMQ-RK2 are compared with exact solution respectively at final time $t = 20$.}\label{duff}
\end{figure}
\section{Conclusions}
In this study, we have introduced MQ-RBF and IMQ-RBF-based Runge-Kutta methods, which demonstrate improved accuracy orders compared to classical Runge-Kutta schemes. A detailed analysis of the convergence behavior, stability regions, and local convergence order was performed through five numerical tests. The results indicate that although the IMQ-RBF method incurs higher computational complexity, it consistently outperforms the MQ-RBF method in terms of both numerical error and convergence order. Furthermore, implementing fourth-stage, fifth-order Runge-Kutta methods with MQ and IMQ radial basis functions (RBFs) on standard CPU architectures presents significant memory challenges, complicating the search for optimal shape parameters. To address these limitations, the development of memory-efficient MQ and IMQ RBF-based fourth-order, five-stage Runge-Kutta methods is currently underway and will be presented in future work.\\
\newline
\noindent{\bf Acknowledgments }
The author Shipra Mahata is supported by UGC, India (NTA
Ref No. 211610046691).\\
\bibliographystyle{elsarticle-num}  
\bibliography{references}
\end{document}